\documentclass{amsart}
\usepackage{style}

\title[Lanczos with compression for Lyapunov equations]{Lanczos with compression for symmetric matrix Lyapunov equations}

\author[A. A. Casulli]{Angelo A. Casulli$^\thankssymb{1}$}
\thanks{$^\thankssymb{1}$Gran Sasso Science Institute (GSSI), L'Aquila, (\href{mailto:angelo.casulli@gssi.it}{\nolinkurl{angelo.casulli@gssi.it}})}

\author[F. Hrobat]{Francesco Hrobat$^\thankssymb{2}$}
\thanks{$^\thankssymb{2}$Scuola Normale Superiore (SNS), Pisa (\href{mailto:francesco.hrobat@sns.it}{\nolinkurl{francesco.hrobat@sns.it})}}

\author[D. Kressner]{Daniel Kressner$^\thankssymb{3}$}
\thanks{$^\thankssymb{3}$École Polytechnique Fédérale de Lausanne (EPFL), Lausanne (\href{mailto:daniel.kressner@epfl.ch}{\nolinkurl{daniel.kressner@epfl.ch})}}

\DeclareFontFamily{U}{mathx}{}
\DeclareFontShape{U}{mathx}{m}{n}{<-> mathx10}{}
\DeclareSymbolFont{mathx}{U}{mathx}{m}{n}
\DeclareMathAccent{\widecheck}{0}{mathx}{"71}

\begin{document}

\begin{abstract}
This work considers large-scale Lyapunov matrix equations of the form $AX + XA = \vec{c}\vec{c}^T$,
where $A$ is a symmetric positive definite matrix and $\vec{c}$ is a vector. Motivated by the need to solve such equations in a wide range of applications, various numerical methods have been developed to compute low-rank approximations of the solution matrix $X$. In this work, we focus on the Lanczos method, which has the distinct advantage of requiring only matrix-vector products with $A$, making it broadly applicable. However, the Lanczos method may suffer from slow convergence when $A$ is ill-conditioned, leading to excessive memory requirements for storing the Krylov subspace basis generated by the algorithm. To address this issue, we propose a novel compression strategy for the Krylov subspace basis that significantly reduces memory usage without hindering convergence. This is supported by both numerical experiments and a convergence analysis. Our analysis also accounts for the loss of orthogonality due to round-off errors in the Lanczos process.

\end{abstract}

\maketitle

\section{Introduction}
Lyapunov matrix equations take the form 
$AX + XA^T = C$ for given matrices $A,C \in \R^{N \times N}$ and an unknown $X \in \R^{N \times N}$. During the last decades, a range of highly efficient solvers for such linear matrix equations have been developed; see~\cite{Sim} for an overview. In this work, we consider the particular case when $A$ is symmetric positive definite, and $C$ is symmetric positive semi-definite and of low rank. By the superposition principle, we may in fact assume that $C$ has rank $1$, that is, there is a vector $\vec{c} \in \R^N$ such that $C = \vec{c}\vec{c}^T$. It is well known that any such \emph{symmetric Lyapunov matrix equation}
\begin{equation}\label{eqn: lowlyapintro}
AX+XA = \vec{c}\vec{c}^T
\end{equation}
has a unique solution $X$, which is symmetric positive semi-definite.

Additionally, we suppose that $A$ is a large, data-sparse matrix, such that both the storage of $A$ and matrix-vector products with $A$ are relatively cheap, while -- for example -- the diagonalization of $A$
is computationally infeasible. Such large-scale Lyapunov equations arise in a number of applications, including control theory~\cite{Dat, Gaj}, model order reduction~\cite{Ant,Ben, Ben2}, as well as structured discretizations of  elliptic partial differential equations~\cite{Pal}.

Most known methods for solving~\cref{eqn: lowlyapintro} in the large-scale setting exploit the fact that the singular values of $X$ decay quickly to zero \cite{Bec3,Gru}. In turn, this makes it possible to aim at computing a memory-efficient, low-rank approximation of $X$. In particular, popular rational methods, such as implicit ADI and \emph{rational} Krylov subspace methods~\cite{MR2557293,Cas2, MR1920565,MR1742324}, are known to converge rapidly to accurate low-rank approximations of $X$. A major limitation of these approaches is that they require the solution of a shifted linear system with $A$ in every iteration, which may become expensive or even infeasible, especially when $A$ is only given implicitly in terms of its action on a vector.

When $A$ is accessed through matrix-vector products only, it is natural to consider (polynomial) Krylov subspace methods~\cite{Jai,Saa}. For symmetric $A$, the Lanczos process~\cite{Golub2013} constructs an orthonormal basis $\b Q_M$ of the Krylov subspace
\[
\pol_M(A,\vec{c}) = \{\vec{c}, A\vec{c}, \dots,A^{M-1}\vec{c} \}, \quad M \ll N,
\]
using a short-term recurrence. This process also returns the 
tridiagonal matrix $\b T_M := \b Q_M^TA \b Q_M$. The Lanczos method applied to the symmetric Lyapunov equation~\cref{eqn: lowlyapintro}
produces the approximation $X \approx \b Q_MX_M\b Q_M^T$ (in factored form),
where $X_M$ satisfies the $M\times M$ projected Lyapunov equation
\begin{equation}\label{eqn: lyapproj}
     \b T_MX_M + X_M\b T_M = \| \vec{c}\|_2^2\vec{e}_1\vec{e}_1^T.
\end{equation}
Thanks to the tridiagonal structure of $\b  T_M$, the solution
of the projected equation~\cref{eqn: lyapproj} can be cheaply computed (by, e.g., ADI), even for relatively large values of $M$.

A major drawback of the Lanczos method, compared to rational Krylov subspace methods, is its slow convergence for ill-conditioned $A$~\cite{Sim2}. In turn, a large value of $M$ may be needed to attain a low approximation error, which has several negative ramifications. The cost of reorthogonalization for ensuring
numerical orthogonality of $\b  Q_M$ grows quadratically with $M$. Even when reorthogonalization is turned off (which delays but does not destroy convergence; see Section~\ref{sec:finiteprec}), the need for storing $\b  Q_M$
in order to be able to return $\b  Q_MX_M\b  Q_M^T$ may impair the Lanczos method significantly.
Strategies for bypassing these excessive memory requirements include 
the two-pass Lanczos method from~\cite{Kre2} and the restarting strategy from~\cite{Kre}.

The two-pass Lanczos method~\cite{Kre2} first performs one pass of the Lanczos process (without reorthogonalization) to construct the matrix $\b  T_M$ \emph{without} storing $\b  Q_M$. After solving the projected equation~\cref{eqn: lyapproj} and computing a low-rank approximation $X_M \approx L_ML_M^T$, a second \emph{identical} pass of the Lanczos process is used to compute the product $\b  Q_ML_M$. As only two vectors are needed to define the Lanczos process, and the numerical ranks of $X_M$ and $X$ are usually similar, the memory required by this method is optimal -- on the level of what is needed anyway to represent the \emph{best} low-rank approximation of $X$. However, this desirable property comes at the expense of (at least) doubling the number of matrix-vector products.

The compress-and-restart strategy proposed in~\cite{Kre}, which also applies to nonsymmetric $A$, initially carries out a limited number of steps of the Krylov subspace method. The resulting approximation is refined by noticing that the  correction also satisfies a Lyapunov equation, with the right-hand side replaced by the residual. The solution to this correction equation is again approximated by carrying out a limited number of steps. These restarts are repeated until the desired accuracy is reached. One issue with this approach is that the rank of the right-hand side snowballs due to restarting. Repeated compression is used in~\cite{Kre} to alleviate this issue but, as we will see in
Section~\ref{sec:numexp}, it can still lead to a significant increase of execution time.

Limited-memory Krylov subspace methods, such as two-pass methods and restarting strategies, have also been proposed in the context of computing a matrix function $f(A)\vec c$; see~\cite{Gut2,Gut3} and the references therein. Recently, an approach has been proposed in~\cite{Cas} that repeatedly applies a rational approximation of $f$ to the tridiagonal matrices generated in the course of the Lanczos process in order to compress the Krylov subspace basis. In this work, we extend this approach from matrix functions to the Lyapunov equation~\cref{eqn: lyapproj}. Our extension relies on a different choice of rational approximation and other modifications of the method from~\cite{Cas} (see~\Cref{sec: algo} for a detailed discussion of the differences).

In a nutshell, Lanczos with compression proceeds as follows for the Lyapunov equation~\cref{eqn: lowlyapintro}:
suppose that the projected equation~\cref{eqn: lyapproj} is solved by a \emph{rational} Krylov subspace method. Typically, the size $k$ of the basis
$\b U_{M,k}$ involved in such a method is much smaller than $M$. The compressed subspace spanned by the $k$ columns of
$\b  Q_M \b U_{M,k}$ contains the essential part of $\b  Q_M$ needed for solving~\cref{eqn: lyapproj}. This simple observation will yield 
our reference method,~\Cref{alg: naive}, introduced and analyzed in~\Cref{sec: premnot}. One obvious flaw of this approach is that 
the product $\b  Q_M \b U_{M,k}$ still requires knowledge of the large Lanczos basis $\b  Q_M$ and, thus, does not decrease memory requirements. 
This flaw will be fixed; by exploiting the tridiagonal structure of $\b T_M$ and implicitly leveraging low-rank updates,
the matrix $\b  Q_M \b U_{M,k}$ is computed on the fly while storing only a small portion of $\b  Q_M$.
This yields our main method,~\Cref{alg: lyapcomp}, which is mathematically equivalent to Algorithm~\ref{alg: naive}.

Our main theoretical contributions are as follows: \Cref{thm:solbound} quantifies the impact of compression on the convergence of the Lanczos method, showing that already a modest number $k$ of Zolotarev poles in the rational approximation make this impact negligible. \Cref{sec:finiteprec} analyzes how the loss of orthogonality in the Lanczos basis, due to roundoff, influences convergence. First,~\Cref{thm:fplanczos} derives an error bound for the Lanczos method itself, which may be of independent interest. Second,~\Cref{thm:fplanczos} derives an error bound for Lanczos with compression~\Cref{thm:fplancomp}. Unless $A$ is extremely ill-conditioned, these error bounds predict convergence close to the convergence \emph{bounds} from~\cite[Section 2.3]{Bec} until the level of roundoff error is reached. 
\section{Lanczos method combined with rational approximation}\label{sec: premnot}

Many methods for solving large-scale Lyapunov equations, including all methods discussed in this work, belong to the general class of subspace projection methods~\cite{Jbi}. Given an orthonormal basis ${Q} \in \mathbb{R}^{N \times M}$ of an $M$-dimensional subspace with $M\ll N$, subspace projection
reduces the original equation~\cref{eqn: lowlyapintro} to the (smaller) $M\times M$ Lyapunov equation
\begin{equation} \label{eq:projectedequation}
 {Q}^T A{Q} Y + Y {Q}^T A {Q} = ({Q}^T \mathbf{c})({Q}^T \mathbf{c})^T.
\end{equation}
Once this projected equation is solved by, e.g., diagonalizing ${Q}^T A{Q}$, one obtains the rank-$M$ approximation $X \approx {Q} Y {Q}^T$.
In the following, we discuss two such subspace projection methods, the standard Lanczos method~\cite{Saa} as well as its combination with a rational Krylov subspace method for solving the projected equation.

\subsection{Lanczos method}

Given a symmetric matrix $A$ and a vector $\vec{c}$, the well-known Lanczos process (summarized in \cref{alg: lanc}) constructs an orthonormal basis $\b{Q}_M$ for the corresponding Krylov subspace $\pol_M(A, \vec{c})$. Additionally, it produces the tridiagonal matrix 
\[ \b{T}_M = \begin{bmatrix}
\alpha_1 & \beta_1 &  &  &  \\
\beta_1 & \alpha_2 & \beta_2 &  & \\
 & \beta_2 & \ddots & \ddots &  \\
 &  & \ddots & \alpha_{M-1} & \beta_{M-1} \\
 &  &  & \beta_{M-1} & \alpha_{M}
\end{bmatrix} \in \R^{M\times M} \]
such that 
\begin{equation} \label{eq:lanczosdecomp}
A\b{Q}_M = \b{Q}_M \b{T}_M + \beta_M \vec{q}_{M+1} \vec{e}_M^T,
\end{equation}
where $\beta_M \in \mathbb{R}$, and $\vec{q}_{M+1} \in \R^N$ is such that $[\b{Q}_M, \vec{q}_{M+1}]$ is an orthonormal basis for $\pol_{M+1}(A, \vec{c})$. Consequently, the coefficients of the projected equation~\cref{eq:projectedequation} are given by $\b{Q}_M^T A \b{Q}_M = \b{T}_M$ and 
$\b{Q}_M^T \vec{c} = \|\vec{c}\|_2 \vec{e}_1$, which matches~\cref{eqn: lyapproj}. We recall that the Lanczos method for the Lyapunov equation~\cref{eqn: lowlyapintro}
is simply \cref{alg: lanc}, followed by computing the solution $X_M$ of the projected equation and returning the approximation $\b{Q}_M X_M \b{Q}_M^T$ in factored form.

In the following, we will refer to one loop of 
\Cref{alg: lanc} in lines 5--9 as a \emph{Lanczos iteration}.
One such iteration produces the next basis vector $\vec{q}_{j+1}$ in a three-term recurrence that only involves the last two vectors 
$\vec{q}_{j-1}$ and $\vec{q}_{j}$.

Throughout the rest of this work, we will assume that $M < N$ and that no breakdown occurs, that is, 
$\pol_M(A, \vec{c})$ has dimension $M$. The presence of a breakdown is a rare and fortunate event, in which case the approximation $\b{Q}_M X_M \b{Q}_M^T$ equals the exact solution.

\begin{algorithm}
\begin{algorithmic}[1]
\caption{Lanczos Process}\label{alg: lanc}
\Require Symmetric $A \in \mathbb{R}^{N \times N}$, $\vec{c} \in \mathbb{R}^N$, number of Lanczos iterations $M$.
\Ensure $\b{Q}_{M+1} = [\vec{q}_1, \dots, \vec{q}_{M+1}]$, an orthonormal basis for $\mathcal{K}_{M+1}(A,\vec{c})$; diagonal entries $\{\alpha_1, \dots, \alpha_M\}$ and subdiagonal entries $\{\beta_1, \dots, \beta_{M}\}$ defining $\b{T}_{M}$.

\State $\beta_0 = 0;$
\State $ \vec q_0 = \vec 0;$
\State $\vec{q}_1 \leftarrow \vec{c}/\|\vec{c}\|_2$;
\For {$j = 1, \dots, M$}
    \tikzmark{lanczosStart}
    \State $\vec{w} \leftarrow A\vec{q}_{j} -\beta_{j-1}\vec{q}_{j-1}$;
    \State $\alpha_j \leftarrow \vec{q}_j^T\vec{w}$;
    \State $\vec{w} \leftarrow \vec{w} - \alpha_j\vec{q}_j$;
    \State $\beta_j \leftarrow \|\vec{w}\|_2$;
    \State $\vec{q}_{j+1} \leftarrow \vec{w}/\beta_j$;
    \tikzmark{lanczosEnd}
\EndFor
\end{algorithmic}
\begin{tikzpicture}[overlay,remember picture]
  \coordinate (start) at (pic cs:lanczosStart);
  \coordinate (end) at (pic cs:lanczosEnd);
  \coordinate (braceStart) at ($(start)+(2.4em,-0.2cm)$);
  \coordinate (braceEnd) at ($(end)+(3.5em,-0.1cm)$);
  \draw[decorate,decoration={brace, mirror, amplitude=5pt, raise=2pt}]
    (braceEnd) -- (braceStart)
    node[midway,right=6pt] {Lanczos iteration};
\end{tikzpicture}
\end{algorithm}

\subsection{Rational Krylov subspaces}
Given a general matrix $S \in \R^{M \times M}$, a rational Krylov subspace is constructed from repeatedly solving shifted linear systems with $S$;
see, e.g.,~\cite{Gut,  casThesis} for an introduction.

\begin{definition}[Rational Krylov subspace]  
For $S \in \mathbb{R}^{M \times M}$, consider a list of poles $\vec{\xi}_k = [\xi_1, \dots, \xi_{k}] \in {\mathbb{C}}^k$ that is closed under complex conjugation and does not contain any eigenvalue of $S$.
Given a block vector $B \in \mathbb{R}^{N \times \ell}$, the corresponding rational Krylov subspace is defined as  

\[
\rat(S, B, \vec{\xi}_k) := \operatorname{colspan} \Big\{ r(S) B,\Big|\,r(z) = \frac{p(z)}{q(z)}, \quad p \in \mathcal{P}_{k-1} \Big\},
\]
where  
\[
q(z) := (z-\xi_1)(z-\xi_2) \cdots (z-\xi_k),
\]
and $\mathcal{P}_{k-1}$ denotes all real polynomials of degree at most $k-1$.  
\end{definition}

In this work, we will only use rational Krylov subspaces with block size $\ell = 1$ or $\ell = 2$.
Moreover, to simplify the discussion, we will always assume that the dimension of $\rat(S, B,\vec{\xi}_k)$ is equal to $k \ell$. 

An orthonormal basis for the rational Krylov subspace $\rat(S, B, \vec{\xi}_k)$ can be computed using the (block) rational Arnoldi algorithm; see~\cite{Els} and~\cite[Algorithm~1]{casThesis}.

\subsection{Reference Method}

The reference method, that will serve as the basis of our subsequent developments, is a modification of the Lanczos method.
Instead of solving the projected
equation~\cref{eqn: lyapproj} exactly, it uses 
subspace projection with an orthonormal basis $\b U_{M,k}$ for the rational Krylov subspace $\rat(\b{T}_M, \vec{e}_1, \vec{\xi}_k)$ to approximate the solution $X_M$.
The corresponding pseudocode is outlined in \Cref{alg: naive}; suitable choices for the poles $\vec{\xi}_k$ will be discussed in \Cref{sub:zolotarev}. 

\begin{algorithm}
\caption{Reference Method}\label{alg: naive}
    \begin{algorithmic}[1]
    \Require Symmetric positive definite $A \in \mathbb{R}^{N \times N}$, $\vec{c} \in \mathbb{R}^N$, number of Lanczos iterations $M$, and list of $k$ poles $\vec{\xi}_k$ closed under complex conjugation.
    \Ensure Approximation $X_{\mathtt{ref}}$  in factored form to the solution of the Lyapunov equation~\cref{eqn: lowlyapintro}.
    \State Apply Lanczos process (\cref{alg: lanc}) to compute orthonormal basis $\b{Q}_M$ and tridiagonal matrix $\b{T}_M$;
    \State Compute orthonormal basis $\b{U}_{M,k}$ for rational Krylov subspace $\rat(\b{T}_M, \vec{e}_1, \vec{\xi}_k)$;
    \State Solve projected equation $\b{U}_{M,k}^T \b{T}_M \b{U}_{M,k} Y_{M,k} + Y_{M,k} \b{U}_{M,k}^T \b{T}_M \b{U}_{M,k} ={\|\vec{c}\|_2^2} (\b{U}_{M,k}^T \vec{e}_1) (\b{U}_{M,k}^T \vec{e}_1)^T$; 
    \State \Return $X_{\mathtt{ref}} = (\b{Q}_M \b{U}_{M,k}) Y_{M,k} (\b{Q}_M \b{U}_{M,k})^T$.
    \end{algorithmic}
\end{algorithm}

It is easy to see that Algorithm~\ref{alg: naive} is again a subspace projection method with the orthonormal basis $\b{Q}_M \b{U}_{M,k} \in \R^{N\times k}$. Compared to the Lanczos method, this method reduces the computational cost for
solving the projected equation and returns an equally accurate approximation of much lower rank, provided that the poles are well chosen. However, it does \emph{not} address 
the memory issues related to storing the $N\times M$ matrix $\b{Q}_M$, because the entire matrix is needed to compute the matrix $\b{Q}_M \b{U}_{M,k}$, which usually has much fewer columns. In \Cref{sec: algo}, we will circumvent this drawback by modifying
Algorithm~\ref{alg: naive} such that $\b{Q}_M \b{U}_{M,k}$ is computed implicitly.
 
\subsection{Error Bounds}
\label{sub:error}

In this section, we provide error bounds for \cref{alg: naive} that quantify the impact of approximating the projected equation~\cref{eqn: lyapproj} within the Lanczos method by a rational Krylov subspace method. First, we state a known result on the error for the projected equation itself.
\begin{lemma}\label{lem:dru}
Consider Algorithm~\ref{alg: naive} applied to a 
symmetric positive definite matrix $A\in\R^{N\times N}$, with 
the smallest and largest eigenvalues of $A$ denoted by $\lambda_{\min}$ and $\lambda_{\max}$, respectively.
Suppose that
none of the poles $\xi_i$ is in $[\lambda_{\min},\lambda_{\max}]$, and define
\begin{equation} \label{eq:raterr}
\raterr(\vec{\xi}_k, \lambda_{\min},\lambda_{\max}):=
 \max_{z\in[\lambda_{\min},\lambda_{\max}]}\frac{\prod_{i=1}^k |z+\bar{\xi}_i|^2}{\prod_{i=1}^k |z-\xi_i|^2}.
\end{equation}
Then the error between the solution $X_M$
of the projected equation~\cref{eqn: lyapproj} and its approximation $\b{U}_{M,k}Y_{M,k}\b{U}_{M,k}^T$ satisfies
\[
\|X_M - \b{U}_{M,k}Y_{M,k}\b{U}_{M,k}^T\|_F \le \frac{\raterr(\vec{\xi}_k, \lambda_{\min},\lambda_{\max})}{\lambda_{\min}} \,  \, \|\vec{c}\|_2^2.
\]
\end{lemma}

\begin{proof}
This result is a direct consequence of Theorem 4.2 from~\cite{Dru2}, taking into account that the spectrum of $\b{T}_M = \b{Q}_M^T A \b{Q}_M$ is contained in the interval $[\lambda_{\min}, \lambda_{\max}]$.
\end{proof}

We now relate the approximation error of the reference method with the Lanczos method.

\begin{corollary}\label{thm:solbound}
Consider the setting of \Cref{lem:dru}. Then the approximation $X_{\mathtt{ref}}$ returned by Algorithm~\ref{alg: naive} satisfies the error bound
\begin{equation}\label{eqn:solbound}
\|X - X_{\mathtt{ref}}\|_F \le \|X - X_{\mathtt{lan}}\|_F + \frac{\raterr(\vec{\xi}_k, \lambda_{\min},\lambda_{\max})}{\lambda_{\min}}  \|\vec{c}\|_2^2,
\end{equation}
where $X_{\mathtt{lan}} = \b{Q}_M X_M \b{Q}_M^T$ denotes the approximation returned by the Lanczos method.
\end{corollary}
\begin{proof}
By the triangle inequality,
\[
\|X - X_{\mathtt{ref}}\|_F \le \|X - X_{\mathtt{lan}}\|_F + \|X_{\mathtt{lan}} - X_{\mathtt{ref}}\|_F.
\]
Noting that
$
X_{\mathtt{lan}} - X_{\mathtt{ref}} = \b{Q}_M (X_M - \b{U}_{M,k} Y_{M,k} \b{U}_{M,k}^T) \b{Q}_M^T
$
and applying \Cref{lem:dru} concludes the proof:
\begin{equation}  \label{eq:lanref}
 \|X_{\mathtt{lan}} - X_{\mathtt{ref}}\|_F \le \raterr(\vec{\xi}_k, \lambda_{\min},\lambda_{\max}) / \lambda_{\min}\cdot \|\vec{c}\|_2^2.
\end{equation}
\end{proof}
Similar to~\Cref{thm:solbound}, one obtains the following bound that relates the residual norm
of the reference method with Lanczos method:
\begin{equation}\label{eqn:resbound}
\begin{aligned}
\|AX_{\mathtt{ref}} + X_{\mathtt{ref}}A - \vec{c}\vec{c}^T\|_F 
&\le \|AX_{\mathtt{lan}} + X_{\mathtt{lan}}A - \vec{c}\vec{c}^T\|_F \\
&\quad + 2\,\frac{\lambda_{\max}}{\lambda_{\min}} \raterr(\vec{\xi}_k, \lambda_{\min},\lambda_{\max})
\|\vec{c}\|_2^2.
\end{aligned}
\end{equation}

\subsection{Pole selection}
\label{sub:zolotarev}

The bounds from \Cref{sub:error} suggest to choose the poles such that the expression
$\raterr(\vec{\xi}_k, \lambda_{\min},\lambda_{\max})$ defined in~\cref{eq:raterr} 
is minimized. This problem has been extensively studied in the literature on the ADI method~\cite{Ell}, and explicit formulas for the optimal poles — commonly referred to as Zolotarev poles — can be obtained from solving the third Zolotarev problem on real symmetric intervals. In particular, according to \cite[Thm 3.3]{Bec3}, selecting $\vec \xi_k$ as the Zolotarev poles ensures that the quantity $\raterr(\vec{\xi}_k, \lambda_{\min},\lambda_{\max})$ is given by  
\begin{equation}\label{eq:zolnum}
Z_k([-\lambda_{\max},-\lambda_{\min}], [\lambda_{\min},\lambda_{\max}]) := \min_{r \in \mathcal{R}_{k,k}} \frac{\sup_{z \in [\lambda_{\min},\lambda_{\max}]} |r(z)|}{\inf_{z \in [-\lambda_{\max},-\lambda_{\min}]} |r(z)|},
\end{equation}  
where $\mathcal{R}_{k,k}$ denotes the set of rational functions of the form \( p/q \), with \( p, q \in \mathcal P_{k} \).

The quantity \cref{eq:zolnum} is known as the Zolotarev number and decays exponentially to zero as \( k \) increases~\cite{Bec3}. Specifically, we have  the bound   
\begin{equation}\label{eqn:zolbound}
Z_k([-\lambda_{\max},-\lambda_{\min}], [\lambda_{\min},\lambda_{\max}]) \leq 4\left[\exp\left(\frac{\pi^2}{2\log(4\lambda_{\max}/\lambda_{\min})}\right)\right]^{-2k}.
\end{equation}
Thus, the error bound~\cref{eqn:solbound} of~\Cref{thm:solbound} implies
\[
\|X - X_{\mathtt{ref}}\|_F \le \|X - X_{\mathtt{lan}}\|_F + \frac{4}{\lambda_{\min}}\left[\exp\left(\frac{\pi^2}{2\log(4\lambda_{\max}/\lambda_{\min})}\right)\right]^{-2k} \|\vec{c}\|_2^2,
\]  
and an analogous implication holds for the residual bound~\cref{eqn:resbound}.
Given any $\epsilon > 0$, we can thus determine a suitable integer $k$ such that the approximation $X_{\mathtt{ref}}$ returned by the reference method with $k$ Zolotarev poles differs from the approximation $X_{\mathtt{lan}}$ returned by Lanczos method by at most $\epsilon$, in terms of the error and/or residual norms.
Importantly, $k$ grows only logarithmically with respect to $\epsilon^{-1}$ and $\lambda_{\max}/\lambda_{\min}$. Moreover, the
pole selection strategy is independent of the number of iterations $M$ of the Lanczos process.  

\subsection{Computation of the residual}

In practice, the total number of Lanczos iterations $M$ required by~\cref{alg: naive} to achieve a certain accuracy is not known in advance. Following common practice, we use the norm of the 
residual
$AX_{\mathtt{ref}} + X_{\mathtt{ref}} A - \vec{c} \vec{c}^T$
to decide whether to stop~\cref{alg: naive} or continue the Lanczos process. The following result yields a cheap (and tight) bound for estimating this residual norm.

\begin{lemma}\label{lemma:resdecomp}
Consider the setting of \cref{lem:dru}, and let $\b U_{M,k}$ and $Y_{M,k}$ be the matrices produced in line 3 of~\cref{alg: naive}. Then the approximation returned by \cref{alg: naive} satisfies
\begin{align}\label{eqn: err}
\begin{split}
 &\|AX_{\mathtt{ref}} + X_{\mathtt{ref}} A - \vec{c} \vec{c}^T\|_F^2 \\
 & \le  2\beta_M^2\left\lVert \vec{e}_{M}^T\b U_{M,k}Y_{M,k} \right\rVert_2^2 + 2 \left(\,\frac{\lambda_{\max}}{\lambda_{\min}}
\raterr(\vec{\xi}_k, \lambda_{\min},\lambda_{\max})\|\vec{c}\|_2^2\right)^2.
\end{split}
\end{align}
\end{lemma}

\begin{proof}
We first note that range and co-range of the residual $\mathsf{res} := AX_{\mathtt{ref}} + X_{\mathtt{ref}} A - \vec{c} \vec{c}^T$ are contained in $\mathcal{K}_{M+1}(A, \vec{c})$, with the orthonormal basis $\b Q_{M+1} = [\b Q_M, \vec{q}_{M+1}]$.
This allows us to decompose
\begin{align}
 \|\mathsf{res}\|_F^2 &= \|\b Q_M^T\mathsf{res}\,\b Q_M\|_F^2 + 2 \|\vec{q}_{M+1}^T\mathsf{res}\,\b Q_M\|_2^2 + 
 |\vec{q}_{M+1}^T\mathsf{res}\,\vec{q}_{M+1}|^2 \nonumber \\
 &= \|\b Q_M^T ( A( X_{\mathtt{ref}}- X_{\mathtt{lan}} ) + ( X_{\mathtt{ref}}- X_{\mathtt{lan}} ) A )  \b Q_M\|_F^2 + 2 \|\vec{q}_{M+1}^T\mathsf{res}\,\b Q_M\|_2^2 \nonumber  \\
 &\le 2 \lambda_{\max}^2  \|X_{\mathtt{ref}}- X_{\mathtt{lan}} \|_F^2
 + 2 \|\vec{q}_{M+1}^T A \b Q_M \b U_{M,k}Y_{M,k} \|_2^2 \label{eq:aux2}
\end{align}
where the second equality follows from
$\b Q_M^T( AX_{\mathtt{lan}} + X_{\mathtt{lan}} A - \vec{c} \vec{c}^T ) \b Q_M = 0$
and $|\vec{q}_{M+1}^T\mathsf{res}\,\vec{q}_{M+1}| = 0$ (implied by $X_{\mathtt{ref}} \vec{q}_{M+1} = 0$, $\vec{q}_{M+1}^T X_{\mathtt{ref}} = 0$).

To bound the first term in~\cref{eq:aux2}, we use the bound~\cref{eq:lanref} for $\|X_{\mathtt{ref}}- X_{\mathtt{lan}} \|_F$, which leads to the second term in~\cref{eqn: err}. To bound the second term in~\cref{eq:aux2}, we note that the Lanczos decomposition~\cref{eq:lanczosdecomp} implies 
$\vec{q}_{M+1}^T A \b Q_M = \beta_M \vec{e}_{M}^T$, 
which leads to the first term in~\cref{eqn: err}.
\end{proof}

To ensure that Algorithm~\ref{alg: naive} produces a residual norm below a prescribed 
relative tolerance $\texttt{tol}\cdot \|\vec{c}\|_2^2$, the result of \Cref{lemma:resdecomp} suggests to first choose 
$k$ and Zolotarev poles $\vec{\xi}_k$ such that the error bound~\cref{eqn:zolbound} multiplied 
by $\lambda_{\max} / \lambda_{\min}$ remains below $\texttt{tol} / {2}$. Then the Lanczos process 
is carried out until
$\beta_M \left\lVert \vec{e}_{M}^T\b U_{M,k}Y_{M,k} \right\rVert_2 \le \texttt{tol} \cdot \|\vec{c}\|_2^2 / {2}$
is satisfied.
\section{Main algorithm}\label{sec: algo}

The goal of this section is to modify the reference method (\cref{alg: naive}) such that it avoids storing the entire basis $\b{Q}_M$ produced by the Lanczos process. In~\Cref{sec:notation}, we first introduce the necessary notation and provide an intuitive derivation of the algorithm, while theoretical results are presented in \Cref{sec:theoretical} and implementation aspects are discussed in \Cref{sub: praimp}.

In~\cite{Cas}, a compression strategy for the Lanczos method applied to a matrix function $f(A)\vec{c}$ has been presented, which
successively updates an approximation to $f(A)\vec{c}$ every fixed number of Lanczos iterations.
While clearly inspired by~\cite{Cas}, our compression strategy for Lyapunov equations is different. Successive updates of the approximate solution
would significantly increase its rank. Although this increase could be mitigated by repeated low-rank approximation, such a measure might be costly and difficult to justify theoretically. Therefore, instead of updating the approximate solution, we 
employ an update strategy based on the rational Krylov subspace itself. This comes with the additional advantage that,
unlike~\cite[Prop. 4.2]{Cas}, the additional error incurred by compression does not 
depend on the number of cycles.

\subsection{Notation and derivation of algorithm} 
\subsubsection{Partitioning Lanczos basis and tridiagonal matrix into cycles}\label{sec:notation} 
Following~\cite{Cas}, the Lanczos iterations are divided into $s$ cycles as follows:
\begin{itemize}
 \item the first cycle consists of $m + 2k$ iterations, where $m$ is fixed and $k$ is the number of Zolotarev poles;
 \item while each of the remaining $s - 1$ cycles consists of $m$ iterations. 
\end{itemize}
Consequently, the total number of Lanczos iterations performed is $M = sm + 2k$. As we will see below, at most $m + 2k +1$ vectors of length $N$ need to be stored in memory throughout the algorithm.

We use $i = 1,\ldots,s$ to index the cycle. The total number of Lanczos iterations performed until cycle $i$ is given by $im+2k$. Until cycle $i$ the Lanczos process generates $im+2k+1$ basis vectors (which are \emph{not} fully stored) denoted by $[Q_i, \vec q_{im+2k+1}]$, where $Q_i \in \mathbb{R}^{N \times (im+2k)}$ contains the first $im+2k$ columns of $\mathbf{Q}_M$. We let $\widehat{Q}_{i+1}\in \R^{N\times m}$ denote the matrix generated during cycle $i+1$, so that \[
Q_{i+1} = [Q_i, \widehat{Q}_{i+1}], \quad i = 1,\ldots,s-1;
\]
see \Cref{fig: Qi} for an illustration. Note that, because of the three-term recurrence relation, only the last column of $Q_i$ and the vector $\vec q_{im+2k+1}$ are needed to compute $\widehat{Q}_{i+1}$.

\begin{figure}
\begin{center}
\begin{tikzpicture}[scale = 0.7]\label{fig: Qi}

    \node at (-1.1, 2.5) {$\b{Q}_M = $};

    \draw[white] (3.5,0) rectangle (5,5) node[pos=.5] {\textcolor{black}{$\dots$}};
    \fill[red!20] (0,0) rectangle (3.5,5);
    \draw[dotted] (0,0) rectangle (2,5);
    \draw[red!20] (2,0) rectangle (3.5,5) node[pos=.51] {\textcolor{black}{$\widehat{Q}_{i+1}$}};
    \draw[dotted] (0,0) rectangle (2,5) node[pos=.5] {\textcolor{black}{$Q_i$}};
    \draw[white](2,  5+0.2) -- node[above]{\textcolor{black}{$m$}}(3.5 , 5 + 0.3);
    \draw [decorate,decoration={brace,amplitude=2mm, raise=.5mm}] (2, 5) -- ( 3.5, 5);
    \draw[very thick] (0,0) rectangle (5,5);
    \draw[] (0,0) rectangle (3.5,5);
    \draw[white](0,  5+0.2) -- node[above]{\textcolor{black}{$im+2k$}}(2 , 5 + 0.3);
    \draw [decorate,decoration={brace,amplitude=2mm, raise=.5mm}] (0, 5) -- ( 2, 5);

    \node at (6, 2.5 ) {$ = $};

    \draw[white] (10.5,0) rectangle (12,5) node[pos=.5] {\textcolor{black}{$\dots$}};
    \draw[very thick] (7,0) rectangle (12,5);
    \fill[red!40] (7,0) rectangle (10.5,5);
    \draw[] (7,0) rectangle (10.5,5) node[pos=.5] {$Q_{i+1}$};
    \draw[white](7,  5+0.2) -- node[above]{\textcolor{black}{$(i+1)m+2k$}}(10.5 , 5 + 0.3);
    \draw [decorate,decoration={brace,amplitude=2mm, raise=.5mm}] (7, 5) -- ( 10.5, 5);

\end{tikzpicture}
\end{center}
\caption{Graphical representation of the orthonormal basis $Q_{i+1}$ computed until cycle $i+1$.}
\end{figure}
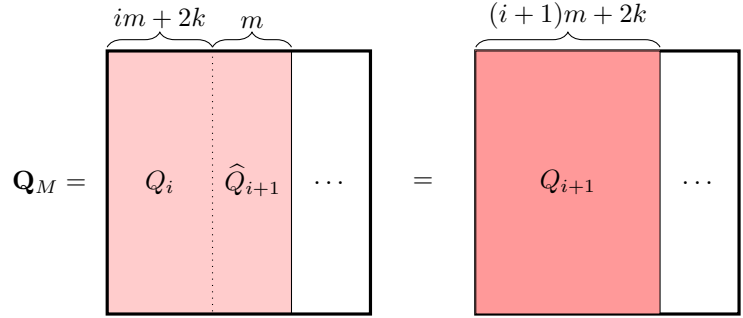

The tridiagonal matrix obtained from the Lanczos process until cycle $i$ is denoted by $T_i := Q_i^T A Q_i$. Additionally, the $m \times m$ tridiagonal matrix generated during cycle $i+1$ is denoted by $\widehat{T}_{i+1}:=\widehat{Q}_{i+1}^T A \widehat{Q}_{i+1}$.  
Note that  
\begin{equation}\label{eqn:Ti}  
T_{i+1} = \begin{bmatrix} 
T_i & \beta_{im+2k} \vec{e}_{im+2k} \vec{e}_1^T \\  
\beta_{im+2k} \vec{e}_1 \vec{e}_{im+2k}^T & \widehat{T}_{i+1}  
\end{bmatrix},  
\end{equation}  
and both $T_i$ and $\widehat{T}_{i+1}$ are principal submatrices of the full tridiagonal matrix $\b{T}_M$; see \Cref{fig: Ti}.  

\begin{figure}
\begin{center}
\begin{tikzpicture}[scale = 0.5]\label{fig: Ti}
\def\scale{0.19};
\def\band{0.15};
\def\n{4};

\node at (-2*\n-0.5 , 0) {$\b{T}_M = $};

\draw[white] (2*\n*4*\scale - \n, -2*\n*4*\scale + \n) rectangle (\n,-\n) node[pos=.4] {\textcolor{black}{$\ddots$}};
 
\foreach \i in {3} {
  \pgfmathsetmacro\k{20}
  \draw[fill = red!\k] (\n*\i*\scale + \n*\i*\scale - \n, \n*\i*\scale + \n - \n*\i*\scale) rectangle (-\n*\i*\scale + \n*\i*\scale - \n, -\n*\i*\scale + \n - \n*\i*\scale) node[pos=.5] {$T_i$};
}

\foreach \i in {3} {
  \pgfmathsetmacro\k{20}

  \draw[very thick] (\n*\i*\scale + \n*\i*\scale - \n, \n*\i*\scale + \n - \n*\i*\scale)--(-\n*\i*\scale + \n*\i*\scale - \n, \n*\i*\scale + \n - \n*\i*\scale)--(-\n*\i*\scale + \n*\i*\scale - \n, -\n*\i*\scale + \n - \n*\i*\scale)--(\n*\i*\scale + \n*\i*\scale - \n, -\n*\i*\scale + \n - \n*\i*\scale)--cycle;
}

\node at (+\n*3*\scale + \n*3*\scale - \n +\scale, -\n*3*\scale + \n - \n*3*\scale + \scale) [circle,fill,inner sep=1pt]{};

\node at (+\n*3*\scale + \n*3*\scale - \n -\scale, -\n*3*\scale + \n - \n*3*\scale - \scale) [circle,fill,inner sep=1pt]{};

\draw[fill = red!20] (2*\n*3*\scale - \n, -2*\n*3*\scale + \n) rectangle (2*\n*4*\scale - \n, -2*\n*4*\scale + \n) node[pos = .5] {$\widehat{T}_{i+1}$};
\draw[very thick] (2*\n*3*\scale - \n, -2*\n*3*\scale + \n)--(2*\n*4*\scale - \n, -2*\n*3*\scale + \n)--(2*\n*4*\scale - \n, -2*\n*4*\scale + \n)--(2*\n*3*\scale - \n, -2*\n*4*\scale + \n)--cycle;

\draw[very thick] (\n,\n)--(-\n,\n)--(-\n,-\n)--(\n,-\n)--cycle;

\draw[white](- \n - 0.3, -\n*4*\scale + \n - \n*4*\scale) -- node[left]{\textcolor{black}{$m$}}(- \n - 0.3, -\n*3*\scale + \n - \n*3*\scale);

\draw [decorate,decoration={brace,amplitude=2mm, raise=.5mm}] (- \n, -\n*4*\scale + \n - \n*4*\scale) --(- \n, -\n*3*\scale + \n - \n*3*\scale);

\draw[white](\n*4*\scale - \n + \n*4*\scale,  \n + 0.4) -- node[above]{\textcolor{black}{$m$}}(\n*3*\scale - \n + \n*3*\scale, \n + 0.4);

\draw [decorate,decoration={brace,amplitude=2mm, mirror, raise=.5mm}] (\n*4*\scale - \n + \n*4*\scale, \n) -- (\n*3*\scale - \n + \n*3*\scale, \n);

\draw[white](- \n - 0.4, -\n*3*\scale + \n - \n*3*\scale) -- node[left]{\textcolor{black}{$im+2k$}}(- \n - 0.4,  \n);

\draw [decorate,decoration={brace,amplitude=2mm, raise=.5mm}] (- \n, -\n*3*\scale + \n - \n*3*\scale) --(- \n,  + \n );

\draw[white](\n*3*\scale - \n + \n*3*\scale,  \n + 0.3) -- node[above]{\textcolor{black}{$im+2k$}}(- \n , \n + 0.3);

\draw [decorate,decoration={brace,amplitude=2mm, mirror, raise=.5mm}] (\n*3*\scale - \n + \n*3*\scale, \n) -- ( -\n, \n);

\draw[white] (2*\n*4*\scale +\n +2, -2*\n*4*\scale + \n) rectangle (\n+\n+\n+2,-\n) node[pos=.4] {\textcolor{black}{$\ddots$}};

\draw[very thick] (3*\n + 2,\n)--(\n+2,\n)--(\n+2,-\n)--(3*\n+2,-\n)--cycle;

\node at (\n + 1, 0) {$ = $};

\foreach \i in {4} {
  \pgfmathsetmacro\k{40}
  \draw[fill = red!\k] (\n*\i*\scale + \n*\i*\scale +2+\n, \n*\i*\scale + \n - \n*\i*\scale) rectangle (-\n*\i*\scale + \n*\i*\scale +2+\n, -\n*\i*\scale + \n - \n*\i*\scale) node[pos=.5] {$T_{i+1}$};
}

\foreach \i in {4} {

  \draw[very thick] (\n*\i*\scale + \n*\i*\scale + \n +2, \n*\i*\scale + \n - \n*\i*\scale)--(-\n*\i*\scale + \n*\i*\scale + \n +2, \n*\i*\scale + \n - \n*\i*\scale)--(-\n*\i*\scale + \n*\i*\scale + \n +2, -\n*\i*\scale + \n - \n*\i*\scale)--(\n*\i*\scale + \n*\i*\scale + \n +2, -\n*\i*\scale + \n - \n*\i*\scale)--cycle;
  
}

\draw[white](\n*4*\scale \n + 2 +\n*4*\scale,  \n + 0.3) -- node[above]{\textcolor{black}{$(i+1)m+2k$}}(2*\n +2 , \n + 0.3);

\draw [decorate,decoration={brace,amplitude=2mm, mirror, raise=.5mm}] (\n*4*\scale + \n + 2+\n*4*\scale, \n) -- ( \n +2 , \n);

\end{tikzpicture}
\end{center}
\caption{Graphical representation of the tridiagonal matrices $T_{i+1}$ generated until cycle $i+1$.}
\end{figure}
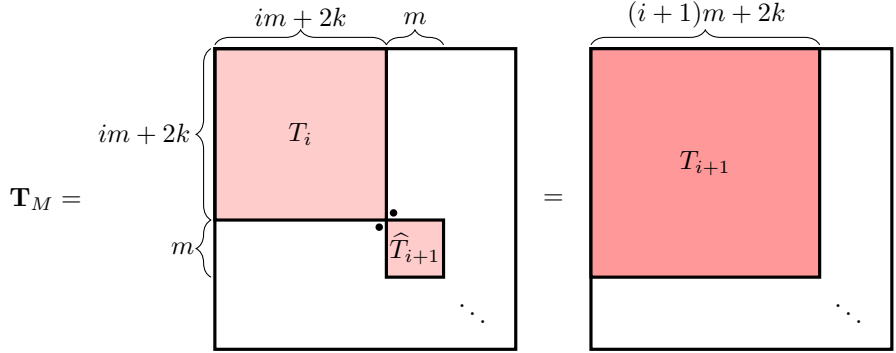

\subsubsection{Recursive computation of rational Krylov subspace bases} 

Let us recall that \cref{alg: naive} performs an \emph{a posteriori} compression of the Lanczos basis
$\mathbf{Q}_M$ by multiplying it with a basis of the rational Krylov subspace $\rat(\mathbf{T}_M,\vec e_1, \vec{\xi}_k)$. Turning this into an \emph{on the fly} compression will allow us to avoid storing $\mathbf{Q}_M$. For this purpose, we will define two recursive sequences of rational Krylov subspace bases.
Our construction \emph{implicitly} leverages a rational variant of the Sherman–Morrison–Woodbury formula~\cite{Bec2,Bernstein2000}, using
that the diagonal blocks $T_{i}$ and $\widehat T_{i+1}$ in $T_{i+1}$ are coupled via a rank-$2$ update.

We will now define, recursively, a primary sequence $U_i$ that is used for compressing $\mathbf{Q}_M$, and
an auxiliary sequence $W_i$ that is used for keeping track of updates to rational functions of $T_i$.
For this purpose, we first choose orthonormal bases $W_1 \in \mathbb{R}^{(m+2k) \times 2k}$ and $\widetilde{U}_1 \in \mathbb{R}^{2k \times k}$ such that
\begin{equation*}
\mathrm{span}(W_1) = \rat(T_1, [\vec e_1,\vec{e}_{m+2k}], \vec{\xi}_k),
\quad
\mathrm{span}(\widetilde U_1) = \rat(W_1^T T_1 W_1, W_1^T \vec{e}_1, \vec{\xi}_k),
\end{equation*}
and we set $U_1 := W_1 \widetilde{U}_1$.

To proceed from $i$ to $i+1$ for $i \ge 1$, we first update $W_i$ as follows:
\begin{equation}\label{eqn: tildeW}
W_{i+1} := \begin{bmatrix} W_i & 0 \\ 0 & I_m \end{bmatrix} \widetilde{W}_{i+1},
\end{equation}
where $\widetilde{W}_{i+1}\in \R^{(m+2k)\times 2k}$ is an orthonormal basis of 
$
 \rat\Big( S_{i+1}, \begin{bmatrix} \vec{w}_{i} & 0\\ 0 & \vec{e}_m \end{bmatrix}, \vec{\xi}_k \Big)
$
with 
 \begin{equation}\label{eqn:Si}
S_{i+1} : = \begin{bmatrix} W_i & 0 \\ 0 & I_m \end{bmatrix}^T T_{i+1} \begin{bmatrix} W_i & 0 \\ 0 & I_m \end{bmatrix} \in \R^{(m+2k)\times (m+2k)}, \quad
\vec{w}_i := W_i^T \vec{e}_1 \in \R^{2k}.
\end{equation}
We then obtain $U_{i+1}$, the next element of the primary sequence, as follows:
\begin{equation}\label{eqn: tildeU}
U_{i+1} := \begin{bmatrix} W_i & 0 \\ 0 & I_m \end{bmatrix} \widetilde{W}_{i+1} \widetilde{U}_{i+1} = W_{i+1} \widetilde{U}_{i+1},
\end{equation}
where $\widetilde{U}_{i+1}\in \R^{2k\times k}$ is an orthonormal basis of
$\rat\big(\widetilde{S}_{i+1}, \vec{w}_{i+1}, \vec{\xi}_k\big)$ with
\begin{equation}\label{eqn:tildeSi}
\widetilde{S}_{i+1} := \widetilde{W}_{i+1}^T S_{i+1} \widetilde{W}_{i+1} \in \R^{2k\times 2k}.
\end{equation}

\cref{prop:tildeW} below shows that the elements $U_i$ of the primary sequence, constructed as described above, from orthonormal bases
for $\rat(T_i,\vec e_1, \vec{\xi}_k)$, $i = 1,\ldots, s$. 
In particular, $U_s$ matches our desired rational Krylov subspace:
\[
 \mathrm{span}(U_s) = \mathrm{span}(\mathbf{U}_{M,k}) = \rat(\mathbf{T}_M,\vec e_1, \vec{\xi}_k).
\]
Because of $\mathbf{Q}_M = Q_s$, it follows that 
$\mathrm{span}(\mathbf{Q}_M \mathbf{U}_{M,k})=\mathrm{span}(Q_s U_s)$ and, hence, 
we can replace $\mathbf{Q}_M \mathbf{U}_{M,k}$ by $Q_s U_s$ for our purposes.

\subsubsection{Recursive computation of $Q_s U_s$ and residual estimation}

By the discussion above, our main objective is to compute the $N\times k$ matrix $Q_s U_s$, without actually storing the (large) $N\times M$ matrix $Q_s$.
For this purpose, we first compute $Q_1 W_1$ and then update
\begin{equation} \label{eq:upQW}
  Q_{i+1}W_{i+1} = [Q_iW_i,\, \widehat{Q}_{i+1}]\widetilde{W}_{i+1} \in \R^{N\times 2k}, \quad i = 1,\ldots,s-1,
\end{equation}
in accordance with~\cref{eqn: tildeW}.

In each cycle, we thus need to hold exactly $m+2k+1$ vectors of length $N$: the $2k$ columns of the compressed
matrix $Q_iW_i$, the $m$ columns of the newly produced matrix $\widehat{Q}_{i+1}$ and the Lanczos vector $\vec q_{(i+1)m+2k+1}$ which is required to extend the Lanczos basis in the next cycle. In the final cycle $s$, we compute the desired matrix
\begin{equation} \label{eq:upQU}
 Q_{s}U_{s} = [Q_{s-1} W_{s-1},\, \widehat{Q}_{s}](\widetilde{W}_{s}\widetilde{U}_{s}),
\end{equation}
in accordance with~\cref{eqn: tildeW,eqn: tildeU}.

As the total number of Lanczos iterations / cycles is usually not known in advance, we will use the update~\cref{eq:upQW}
until a residual estimate indicates that the algorithm can be terminated, in which case~\cref{eq:upQU}
is computed. Using the result of Lemma~\ref{lemma:resdecomp}, the residual norm after the $i$th cycle can be estimated without needing access to the full matrix $U_i$. Specifically, the first term in the bound~\cref{eqn: err} can be computed using the relation
\begin{equation} \label{eqn:comp-res}
\beta_{im+2k} \left\lVert \vec{e}_{im+2k}^T U_i Y_i \right\rVert_2 =
\beta_{im+2k} \big\| \vec{e}_{im+2k}^T W_i \widetilde U_i Y_i  \big\|_2
= \beta_{im+2k} \big\| \vec{e}_{m+2k}^T \widetilde{W}_i \widetilde{U}_i  Y_i \big\|_2,
\end{equation}  
where the matrix $Y_i$ satisfies the Lyapunov equation  
\begin{equation*}
(\widetilde{U}_i^T \widetilde{S}_i \widetilde{U}_i) Y_i + Y_i (\widetilde{U}_i^T \widetilde{S}_i \widetilde{U}_i) = \| \vec{c} \|_2^2 \cdot  (\widetilde{U}_i^T \vec{w}_i)(\widetilde{U}_i^T \vec{w}_i)^T.
\end{equation*}

\subsubsection{Recursive computation of $S_i$}

Because the size of $W_i \in \R^{(im+2k) \times 2k}$ grows with $i$, its explicit use and storage is best avoided. This matrix is needed in the update~\cref{eqn:Si} of $S_i$.

Noting that $ W_i^T \vec{e}_{im+2k} = \widetilde{W}_i^T \vec{e}_{m+2k} $, using the definition~\cref{eqn:tildeSi} of $\widetilde{S}_i$, it follows that the matrix $ S_{i+1} $ can be computed as  
\begin{equation*}
S_{i+1}  =
\begin{bmatrix}
\widetilde{S}_i & \beta_{im+2k}(\widetilde{W}_i^T \vec{e}_{m+2k})\vec{e}_1^T\\  
\beta_{im+2k} \vec{e}_1 (\widetilde{W}_i^T \vec{e}_{m+2k})^T & \widehat{T}_{i+1}  
\end{bmatrix}.
\end{equation*}
This will allow us to implement our algorithm without storing $W_i$.

\subsection{Theoretical results}\label{sec:theoretical}

In the following, we present the theoretical results required to justify our algorithm. Our main goal is to prove that the matrix $U_{i+1}$ is an orthonormal basis for $\rat(T_{i+1}, \vec{e}_1, \vec{\xi}_k)$.

The next theorem provides a low-rank update formula for evaluating rational matrix functions. We present the specific result required for this work; more general results can be found in \cite{Bec2}, \cite[Sec 2.2]{Cas}.

\begin{theorem}\label{thm: lowrankupdate}
For a list of poles $\vec{\xi}_k \in \mathbb{C}^k$ closed under complex conjugation, set $q(z) = (z-\xi_1) \cdots (z-\xi_k)$
and consider a rational function $r = p/q$ for some $p \in \mathcal P_{k-1}$.
Let $V_i$ be an orthonormal basis of $\rat(T_{i},[\vec{e}_1, \vec{e}_{im+2k}],\vec{\xi}_k)$. Then there exists a matrix $M_{i+1}(r) \in \mathbb{R}^{(m+2k) \times ((i+1)m+2k)}$ such that

\[
r(T_{i+1}) = \begin{bmatrix} r(T_{i}) & 0 \\ 0 & 0 \end{bmatrix} + \begin{bmatrix} V_i & 0 \\ 0 & I_m \end{bmatrix} M_{i+1}(r),
\]
provided that the nested tridiagonal matrices $T_{i}, T_{i+1}$ defined 
in~\cref{eqn:Ti}  
do not have eigenvalues that are contained in $\vec{\xi}_k$.
\end{theorem}

\begin{proof} Applying \cite[Corollary 2.6]{Cas} to the partitioning~\cref{eqn:Ti}, we can express the first $im + 2k$ rows of $r(T_{i+1})$ as  
\[
\begin{bmatrix} r(T_{i}) & 0 \end{bmatrix} + Z_i R_i(r),
\]  
where $Z_i \in \mathbb{R}^{(im+2k) \times k}$ is an orthonormal basis of $\rat(T_i, \vec{e}_{im+2k}, \vec{\xi}_k)$ and $R_i(r) \in \mathbb{R}^{k \times ((i+1)m + 2k)}$. Because of $\operatorname{span}(V_i) \supseteq \rat(T_i, \vec{e}_{im+2k}, \vec{\xi}_k)$, this implies the result.
\end{proof}

\Cref{thm: lowrankupdate} allows us to establish a connection between the rational Krylov subspaces involved in the $i$th and $(i+1)$th cycles of our algorithm.

\begin{corollary}
Under the assumptions of~\Cref{thm: lowrankupdate}, it holds that
\begin{equation}\label{eqn: inclusions} 
\rat(T_{i+1}, \vec{e}_1, \vec{\xi}_k) \subseteq\rat(T_{i+1},[\vec{e}_1, \vec{e}_{(i+1)m+2k}],\vec{\xi}_k)  \subseteq \mathrm{span}\left( \begin{bmatrix} V_i & 0 \\ 0 & I_m \\ \end{bmatrix} \right).
\end{equation}
\end{corollary}
\begin{proof}
The first inclusion holds by the definition. For the second inclusion, we utilize the result of \cref{thm: lowrankupdate}, which implies for a rational function $r$ (in the sense of the theorem) that
\[
r(T_{i+1}) \vec{e}_1 \in \begin{bmatrix} r(T_{i}) \vec{e}_1 \\ 0 \end{bmatrix} + \mathrm{span}\left(\begin{bmatrix} V_i & 0 \\ 0 & I_m \\ \end{bmatrix} \right)
 \subseteq \mathrm{span}\left( \begin{bmatrix} V_i & 0 \\ 0 & I_m \\ \end{bmatrix} \right),
\]
and
\[
r(T_{i+1}) \vec{e}_{(i+1)m+2k} \in \mathrm{span}\left( \begin{bmatrix} V_i & 0 \\ 0 & I_m \\ \end{bmatrix} \right).
\]
Therefore, the result follows from the definition of $\rat(T_{i+1},[\vec{e}_1, \vec{e}_{(i+1)m+2k}],\vec{\xi}_k)$.
\end{proof}

The following proposition states the desired main theoretical result.

\begin{proposition}\label{prop:tildeW}
With the notation introduced above, suppose that the 
the nested tridiagonal matrices $T_{1}, \ldots, T_{s}$ 
do not have eigenvalues that are contained in $\vec{\xi}_k$. Then 
the matrices $W_{i}$ and $U_{i}$ are orthonormal bases for the rational Krylov subspaces $\rat(T_{i},[\vec{e}_1, \vec{e}_{im+2k}],\vec{\xi}_k)$ and $\rat(T_{i}, \vec{e}_1, \vec{\xi}_k)$, respectively, for $i = 1,\ldots, s$.
\end{proposition}
\begin{proof}
We proceed by induction on $i$. For $i = 1$, the matrix $W_1$ is an orthonormal basis of $\rat(T_{1},[\vec{e}_1, \vec{e}_{m+2k}],\vec{\xi}_k)$ by definition, while the claim for $U_1 = W_1 \widetilde{U}_1$ follows from \cite[Proposition~2.3]{Cas}, noting that $\rat(T_{1},\vec{e}_1,\vec{\xi}_k) \subseteq \mathrm{span}(W_1)$.

Assume now that the claim holds for $W_i$ and $U_i$, and let us prove it for $W_{i+1}$ and $U_{i+1}$.
By the induction hypothesis, the inclusions \cref{eqn: inclusions} hold with $V_i = W_i$. Then, by the second inclusion in \cref{eqn: inclusions}, \cite[Proposition~2.3]{Cas} ensures that the matrix $W_{i+1}$ defined in~\cref{eqn: tildeW} forms an orthonormal basis for $\rat(T_{i+1},[\vec{e}_1, \vec{e}_{(i+1)m+2k}],\vec{\xi}_k)$. Similarly, applying \cite[Proposition~2.3]{Cas} to the first inclusion in \cref{eqn: inclusions} guarantees that an orthonormal basis for $\rat(T_{i+1}, \vec{e}_1, \vec{\xi}_k)$ is given by $W_{i+1} \widetilde{U}_{i+1}$, which equals $U_{i+1}$.
\end{proof}

\subsection{Practical implementation}\label{sub: praimp}

The procedure described above for solving the symmetric Lyapunov equation~\cref{eqn: lowlyapintro} is summarized in \Cref{alg: lyapcomp}.

In practice, the poles $\vec{\xi}_k$ are chosen as Zolotarev poles, with $k$ such that the error bound~\cref{eqn:zolbound}, multiplied by $\lambda_{\max} / \lambda_{\min}$, remains below $\mathtt{tol} /{2}$, where $\mathtt{tol}\cdot \lVert\vec c\rVert_2^2$ is a prescribed tolerance on the residual norm. If no bounds for the extremal eigenvalues of $A$ are known a priori, we estimate them in an ad hoc fashion, by computing the minimum and maximum eigenvalues of the projected matrix $T_1$, obtained before the first compression, and multiplying them by $0.1$ and $1.1$, respectively. To ensure a good approximation of the eigenvalues, we also perform full reorthogonalization during the first Lanczos cycle. We assess whether the residual norm falls below $\mathtt{tol} \cdot \lVert \vec{c} \rVert_2^2$ at the end of each cycle by checking if \cref{eqn:comp-res} drops below $\mathtt{tol}\cdot\lVert \vec c\rVert_2^2/{2}$. 

During the algorithm, orthonormal bases for rational Krylov subspaces are computed using the block rational Arnoldi algorithm.

\begin{algorithm}
\caption{Lanczos with compression for symmetric Lyapunov (\texttt{compress})}
\begin{algorithmic}[1]\label{alg: lyapcomp}
\Require Symmetric positive definite $A \in \mathbb{R}^{N \times N}$, $\vec{c} \in \mathbb{R}^N$, list of $k$ poles $\vec{\xi}_k$ closed under complex conjugation and relative tolerance $\mathtt{tol}$.
\Ensure Approximation $X_{\mathtt{ref}}$  in factored form to the solution of the Lyapunov equation~\cref{eqn: lowlyapintro} .
\State Perform $m+2k$ Lanczos iterations (\cref{alg: lanc}) to compute orthonormal basis $[Q_1,\vec q_{m+2k+1}]$ of $\mathcal{K}_{m+2k+1}(A,\vec c)$, and $(m+2k)\times (m+2k)$ tridiagonal matrix $T_1$;
\State Compute orthonormal basis $W_1$ of $\rat(T_1, [\vec{e}_1, \vec{e}_{(m+2k)}], \vec{\xi}_k)$ and set $\widetilde{W}_1 = W_1$; 
\State Compute $\widetilde{S}_1 = W_1^T T_1 W_1$, $\vec{w}_1 = W_1^T\vec{e}_1$ and orthonormal basis $\widetilde{U}_1$ of $\mathcal{Q}(\widetilde{S}_1, \vec w_1, \vec \xi_k)$;
\State Compute $Y_1$ as solution of \begin{equation*}(\widetilde{U}_1^T\widetilde{S}_1\widetilde{U}_1) Y_1 + Y_1(\widetilde{U}_1^T\widetilde{S}_1\widetilde{U}_1)= \| \vec{c}\|_2^2(\widetilde{U}_1^T\vec{w}_1)(\widetilde{U}_1^T\vec{w}_1)^T \end{equation*}
by diagonalizing $\widetilde{U}_1^T\widetilde{S}_1\widetilde{U}_1$;
\If {residual norm is smaller than $\mathtt{tol}\cdot \| \vec{c} \|_2^2$}
    \State \Return $X_{\mathtt{ref}} =  (Q_1W_1\widetilde{U}_1)Y_1 (Q_1W_1\widetilde{U}_1)^T$; 
\EndIf

\State Compute compressed basis $Q_1 W_1 \in \mathbb{R}^{N \times 2k}$;
\State Keep last column $\vec q_{m+2k}$ of $Q_1$ and $\vec{q}_{m+2k+1}$ in memory;
\For{$i = 1, \dots$}
\State Perform $m$ Lanczos iterations, starting from $\vec{q}_{im+2k}$ and $\vec{q}_{im+2k+1}$, to compute
$\widehat T_{i+1}\in \R^{m\times m}$, $\widehat{Q}_{i+1} \in \R^{N \times m}$, and next Lanczos vector $\vec{q}_{(i+1)m+2k+1}$;
\State Compute $S_{i+1} = \begin{bmatrix}
    \widetilde{S}_i &  \beta_{im+2k}(\widetilde{W}_i^T\vec{e}_{m+2k})\vec{e}_1^T \\ \beta_{im+2k}\vec{e}_1(\widetilde{W}_i^T\vec{e}_{m+2k})^T & \widehat{T}_{i+1}
\end{bmatrix}$;
\State Compute orthonormal basis  $\widetilde{W}_{i+1}$ of $\rat\Big( S_{i+1}, \begin{bmatrix} \vec{w}_{i} & 0\\ 0 & \vec{e}_m \\ \end{bmatrix}, \vec{\xi}_k \Big)$;
\State Compute $\widetilde{S}_{i+1} = \widetilde{W}_{i+1}^T S_{i+1}\widetilde{W}_{i+1}$, $\vec{w}_{i+1} = \widetilde{W}_{i+1}^T \begin{bmatrix}\vec{w}_{i}\\ 0 \end{bmatrix}$, and orthonormal basis  $\widetilde{U}_{i+1}$ of $\mathcal{Q}(\widetilde{S}_{i+1}, \vec w_{i+1}, \vec \xi_k)$;
\State Compute $Y_{i+1}$ as solution of \begin{equation*}(\widetilde{U}_{i+1}^T\widetilde{S}_{i+1}\widetilde{U}_{i+1}) Y_{i+1} + Y_{i+1}(\widetilde{U}_{i+1}^T\widetilde{S}_{i+1}\widetilde{U}_{i+1}) = \| \vec{c}\|_2^2(\widetilde{U}_{i+1}^T\vec{w}_{i+1})(\widetilde{U}_{i+1}^T\vec{w}_{i+1})^T\end{equation*} 
by diagonalizing $\widetilde{U}_{i+1}^T\widetilde{S}_{i+1}\widetilde{U}_{i+1}$;
\If {residual norm is smaller than $\mathtt{tol}\cdot \| \vec{c} \|_2^2$}
     \State \Return $X_{\mathtt{ref}} =   ([Q_{i} W_i, \widehat Q_{i+1}]\widetilde {W}_{i+1}\widetilde{U}_{i+1})Y_{i+1}([Q_{i} W_i, \widehat Q_{i+1}]\widetilde {W}_{i+1}\widetilde{U}_{i+1})^T$;
\EndIf
\State Compute compressed basis $Q_{i+1}W_{i+1} = [Q_{i}W_{i}, \widehat{Q}_{i+1}] \widetilde{W}_{i+1}$;
\State Keep last column $\vec q_{(i+1)m+2k}$ of $\widehat Q_{i+1}$ and $\vec{q}_{(i+1)m+2k+1}$ in memory.
\EndFor
\end{algorithmic}
\end{algorithm}

\begin{remark}
\Cref{alg: lyapcomp} directly extends to Sylvester matrix equations of the form  
\[
A_1 X + X A_2 = \vec{c}_1 \vec{c}_2^T,
\]  
where $A_1 \in \mathbb{R}^{N_1 \times N_1}$ and $A_2 \in \mathbb{R}^{N_2 \times N_2}$ are symmetric positive definite matrices.  
In this setting, two separate Lanczos processes are required—one for $A_1$ and one for $A_2$—and two rational Krylov subspaces must be computed iteratively, following the same approach as in \Cref{alg: lyapcomp}.
The poles can still be chosen as Zolotarev poles; however, the intervals defining the Zolotarev function are generally asymmetric in this case. This issue can be addressed by applying a M\"obius transformation to map both intervals onto symmetric ones, as described in \cite[Sec.~3.2]{Bec3}.
The residual norm can be bounded by adapting \cref{lemma:resdecomp}, from which an efficient method to estimate it can also be derived.
\end{remark}
\section{Finite precision behavior of Lanczos method for Lyapunov equations} 
\label{sec:finiteprec}

It is well known that roundoff error severely affects the orthogonality of the basis produced by the Lanczos process. Because
the Lanczos basis is not kept in memory, this issue cannot be mitigated by reorthogonalization in the context of \Cref{alg: lyapcomp}. On the other hand, it is well understood that this loss of orthogonality only delays but does not destroy convergence of 
finite-precision Lanczos methods. Such stability results have been obtained for Lanczos method applied to 
eigenvalue problems~\cite{Pai2}, linear systems~\cite{Dru}, and matrix functions~\cite{Che,Dru}.

In this section, we adapt the analysis of~\cite{Dru} for linear systems to derive results for finite-precision Lanczos method applied to symmetric Lyapunov equations and~\Cref{alg: lyapcomp}.
For this purpose, we let $\epsilon$ denote unit roundoff and define the quantities
\[
    \epsilon_0 = 2(N + 4)\epsilon, \quad 
    \epsilon_1 = 2\left(7 + s \frac{\| |A| \|_2}{\| A \|_2}\right)\epsilon, \quad
    \epsilon_2 = \sqrt{2} \max \{6\epsilon_0, \epsilon_1\},
\]
where $|A|$ denotes the elementwise absolute value, and $s$ denotes the maximum number of nonzeros in any row of $A$. Denote with $\widecheck{\b{Q}}_M,\ \widecheck{\b{T}}_M,\ \widecheck{\vec{q}}_{M+1},\ \widecheck{\beta}_M$ the quantities returned by finite precision Lanczos process \cref{alg: lanc}. Following Paige's analysis~\cite{Pai}, the roundoff error introduced during the Lanczos process leads to a perturbed Lanczos decomposition of the form
\begin{equation}\label{eqn: finprelan}
A\widecheck{\b{Q}}_M = \widecheck{\b{Q}}_M \widecheck{\b{T}}_M + \widecheck{\beta}_M \widecheck{\vec{q}}_{M+1} \vec{e}_M^T + F_M.
\end{equation}
The matrix $\widecheck{\b{T}}_M$ is (still) tridiagonal and symmetric; its spectrum $\Lambda(\widecheck{\b{T}}_M)$ is known to satisfy
\begin{equation} \label{eq:spectrumtm}
 \Lambda(\widecheck{\b{T}}_M) \subset [\lambda_{\min} - M^{5/2}\epsilon_2\| A \|_2, \lambda_{\max} + M^{5/2}\epsilon_2\| A\|_2],
\end{equation}
where we recall that $\lambda_{\min}$, $\lambda_{\max}$ are the smallest/largest eigenvalues of $A$; see~\cite[Thm 2.1]{Dru} and \cite[Eq. (3.48)]{Pai3}. The error term $F_M$ satisfies (under mild conditions on $\epsilon_0$ and $\epsilon_1$)
\begin{equation}
 \label{eq:boundfm}
 \| F_M\|_F \leq \sqrt{M}\epsilon_1\| A \|_2,
\end{equation}
see~\cite[Eq. (21)]{Dru}. According to~\cite[Eq. (22)]{Dru}, it holds that
\begin{equation}\label{eqn: QMnorm} 
\| \widecheck{\b{Q}}_{M}\|_F \leq \sqrt{(1+2\epsilon_0) M}, \quad 
\| \widecheck{\b{Q}}_{M+1}\|_F \leq \sqrt{(1+2\epsilon_0)(M+1)},
\end{equation}
with $\widecheck{\b{Q}}_{M+1} = [\widecheck{\b{Q}}_{M}, \widecheck{\vec{q}}_{M+1}]$.

To simplify considerations, our analysis will focus on the impact of~\cref{eqn: finprelan} on convergence and assume that the rest of the computation (such as the solution of projected Lyapunov equations) is exact. 
In fact, this assumption does not impair our analysis, as the projected Lyapunov equation is solved using a backward stable algorithm. Furthermore, in \Cref{sub: finpreccompress}, full orthogonalization is performed in the rational Arnoldi algorithm.

\subsection{Finite-precision Lanczos without compression}

We start our analysis of finite-precision Lanczos method for the Lyapunov equation~\cref{eqn: lowlyapintro} by assuming that
\begin{equation}\label{eq:assumptionlambdamin}
 \lambda_{\min} > (M+1)^{5/2}\epsilon_2\| A \|_2.
\end{equation}
By~\cref{eq:spectrumtm}, this ensures that $\widecheck{\b{T}}_M$ is positive definite and, hence, the 
projected equation
\[ \widecheck{\b{T}}_MX_M + X_M\widecheck{\b{T}}_M = \| \vec{c} \|_2^2\vec{e}_1\vec{e}_1^T\]
associated with~\cref{eqn: finprelan} has a unique solution $X_M$.

We aim at deriving bounds for the residual 
\begin{equation} \label{eq:fplanres}
 \vec{\rho}_M := A\widecheck{\b{Q}}_MX_M\widecheck{\b{Q}}_M^T + \widecheck{\b{Q}}_MX_M\widecheck{\b{Q}}_M^TA - \vec{c}\vec{c}^T.
\end{equation}
Substituting \cref{eqn: finprelan} into this expression yields
\begin{align}
\vec{\rho}_M &= 
\left( \widecheck{\b{Q}}_M \widecheck{\b{T}}_M 
+ \widecheck{\beta}_M \widecheck{\vec{q}}_{M+1} \vec{e}_M^T 
+ F_M \right) X_M \widecheck{\b{Q}}_M^T \nonumber \\
&\quad + 
\widecheck{\b{Q}}_M X_M 
\left( \widecheck{\b{T}}_M \widecheck{\b{Q}}_M^T 
+ \widecheck{\beta}_M \vec{e}_M \widecheck{\vec{q}}_{M+1}^T 
+ F_M^T \right) \nonumber \\
&\quad - 
\| \vec{c} \|_2^2 
\, \widecheck{\b{Q}}_{M+1} \vec{e}_1 \vec{e}_1^T \widecheck{\b{Q}}_{M+1}^T \nonumber \\
&= 
\widecheck{\b{Q}}_{M+1}
\begin{bmatrix}
\widecheck{\b{T}}_M X_M + X_M \widecheck{\b{T}}_M 
- \| \vec{c} \|_2^2 \vec{e}_1 \vec{e}_1^T 
& 
\widecheck{\beta}_M X_M \vec{e}_M 
\\
\widecheck{\beta}_M \vec{e}_M^T X_M 
& 
0
\end{bmatrix}
\widecheck{\b{Q}}_{M+1}^T \nonumber \\
&\quad + 
F_M X_M \widecheck{\b{Q}}_M^T 
+ \widecheck{\b{Q}}_M X_M F_M^T \nonumber \\
&= 
\widecheck{\b{Q}}_{M+1}
\begin{bmatrix}
0 
& 
\widecheck{\beta}_M X_M \vec{e}_M 
\\
\widecheck{\beta}_M \vec{e}_M^T X_M 
& 
0
\end{bmatrix}
\widecheck{\b{Q}}_{M+1}^T 
+ F_M X_M \widecheck{\b{Q}}_M^T 
+ \widecheck{\b{Q}}_M X_M F_M^T. \label{eqn:rhoM}
\end{align}
Taking the Frobenius norm in~\cref{eqn:rhoM} and using~\cref{eq:boundfm}, we thus obtain
\begin{align}
 \| \vec{\rho} _M\|_F & \le  \sqrt{2} \| \widecheck{\b{Q}}_{M+1}\|_F^2 \|\widecheck{\beta}_M\vec{e}_M^TX_M\|_2 + 2\| \widecheck{\b{Q}}_M \|\|F_M\|_F\| X_M\|_F \nonumber \\
 & \leq \sqrt{2} (1+2\epsilon_0)(M+1) \|\widecheck{\beta}_M\vec{e}_M^TX_M\|_2 + 2\sqrt{1+2\epsilon_0}M \epsilon_1 \|A\|_2 \, \|X_M\|_F \nonumber  \\
 & \leq \sqrt{2} (1+2\epsilon_0)(M+1) \|\widecheck{\beta}_M\vec{e}_M^TX_M\|_2 + \sqrt{1+2\epsilon_0}M\epsilon_1\frac{\|A\|_2\|\vec{c} \|_2^2}{\lambda_{\min}(\widecheck{\b{T}}_M)}, \label{eqn: rhoMbound}
\end{align}
where the last inequality uses \[
\| X_M\|_F \leq \| (\widecheck{\b{T}}_M \otimes I_M + I_M\otimes \widecheck{\b{T}}_M)^{-1}\|_2\|\vec{c}\|_2^2 = \|\vec{c}\|_2^2 / (2 \lambda_{\min}(\widecheck{\b{T}}_M)).                                
                               \]
                               
It remains to discuss the quantity $\|\widecheck{\beta}_M\vec{e}_M^TX_M\|_2$ featuring in the first term of~\cref{eqn: rhoMbound}. For this purpose, we follow~\cite[Sec. 2.3]{Dru} and consider the matrix $\widecheck{\b{T}}_{M+1}$ obtained after one additional iteration of finite-precision Lanczos process. By~\Cref{eq:spectrumtm,eq:assumptionlambdamin}, this matrix is positive definite and, hence, the enlarged projected equation
\begin{equation} \label{eq:enlargedprojected} \widecheck{\b{T}}_{M+1} X_{M+1} + X_{M+1} \widecheck{\b{T}}_{M+1} = \| \vec{c} \|_2^2\vec{e}_1\vec{e}_1^T \end{equation}
also has a unique solution. The quantities
$\widecheck{\b{T}}_M$ and $\widecheck{\beta}_M$ (obtained by the finite-precision Lanczos process) are identical to the corresponding quantities obtained when applying $M$ \emph{exact} Lanczos iterations to $\widecheck{\b{T}}_{M+1}$ with starting vector $\| \vec{c} \|_2 \vec{e}_1$. Now, $\sqrt{2}\|\widecheck{\beta}_M\vec{e}_M^TX_M\|_2$ is the residual norm for the 
approximate solution $\big[ {I_M \atop 0} \big] X_M \big[ {I_M \atop 0} \big]^T$ to~\cref{eq:enlargedprojected} returned by exact 
Lanczos method. This allows us to apply existing convergence results for Krylov subspace methods. In particular,~\cite[Cor. 2.5]{Bec} and~\cite[Eq~(2.11)]{Bec} imply that
\begin{equation} \label{eq:resbound}
 \sqrt{2}\|\widecheck{\beta}_M\vec{e}_M^TX_M\|_2 \le \big(4+4\sqrt{2\kappa_{M+1}}\big)\left( \frac{\sqrt{\widetilde \kappa_{M+1}}-1}{\sqrt{\widetilde\kappa_{M+1}}+1}\right)^M\| \vec{c} \|_2^2,
\end{equation}
where $\kappa_{M+1}$ and $\widetilde\kappa_{M+1}$ are the condition numbers of $\widecheck{\b{T}}_{M+1}$ and 
$\widecheck{\b{T}}_{M+1}+ \lambda_{\min}(\widecheck{\b{T}}_{M+1})I$, respectively. Using~\cref{eq:spectrumtm}, with $M$ replaced by $M+1$, we have the upper bounds
\begin{equation*}
 \kappa_{M+1} \le 
 \frac{\lambda_{\max}+(M+1)^{5/2}\epsilon_2\| A \|_2}{\lambda_{\min} - (M+1)^{5/2}\epsilon_2\| A \|_2}, \quad 
 \widetilde\kappa_{M+1} \le 
 \frac{\lambda_{\max}+\lambda_{\min}}{2\lambda_{\min} - 2(M+1)^{5/2}\epsilon_2\| A \|_2}.
\end{equation*}
Inserting the residual bound~\cref{eq:resbound} into~\cref{eqn: rhoMbound} yields the final result.
\begin{theorem}[Error bound for finite-precision Lanczos method] With notation and assumptions introduced above, the residual $\vec{\rho}_M$ of the approximation
$X_{\mathtt{lan}} = \widecheck{\b{Q}}_MX_M\widecheck{\b{Q}}_M^T$ to the symmetric Lyapunov equation~\cref{eqn: lowlyapintro} obtained from finite-precision Lanczos method satisfies the bound
 \[   \frac{\| \vec{\rho}_M \|_F}{\|\vec{c} \|_2^2} \leq C_1 \left( \frac{\sqrt{{\widetilde\kappa}_{M+1}}-1}{\sqrt{{\widetilde\kappa}_{M+1}}+1}\right)^M  + C_2 \epsilon_1.\]
 with $C_1 = (1+2\epsilon_0)(M+1)\big(4+4\sqrt{2{\kappa}_{M+1}}\big)$ and $C_2 = \frac{\sqrt{1+2\epsilon_0}M\lambda_{\max}}{\lambda_{\min} - M^{5/2}\epsilon_2\| A \|_2}$. \label{thm:fplanczos}
\end{theorem}

Unless $\lambda_{\min}$ is very close to zero, we have that ${\widetilde\kappa}_{M+1} \approx (\lambda_{\max} + \lambda_{\min}) / (2\lambda_{\min})$ and thus the bound of~\Cref{thm:fplanczos} predicts that the residual produced by finite-precision Lanczos method matches the convergence bound from~\cite[Cor. 2.5]{Bec}
until it hits the level of roundoff error.

\subsection{Finite-precision Lanczos with compression}\label{sub: finpreccompress}

We now aim at understanding the impact of roundoff error on Lanczos with compression. Again, we will focus on the effects of the finite-precision Lanczos process and assume that all other computations are carried out exactly. Because~\Cref{alg: lyapcomp} is based on exactly the same Lanczos process, it suffices to study the mathematically equivalent reference method,~\Cref{alg: naive}.

As above, let $\widecheck{\b{T}}_M, \widecheck{\b{Q}}_M$ be the matrices generated by finite-precision Lanczos process and let $\b{U}_{M,k}$ be an orthonormal basis of $\rat(\widecheck{\b{T}}_M,\vec{e}_1,\vec{\xi}_k)$, where $\widecheck{\b{T}}_M$ is the tridiagonal matrix generated by finite-precision Lanczos process. Let $Y_{M,k}$ denote the solution of
\[ \b{U}_{M,k}^T\widecheck{\b{T}}_M\b{U}_{M,k} Y_{M,k} + Y_{M,k}\b{U}_{M,k}^T\widecheck{\b{T}}_M\b{U}_{M,k} = \| \vec{c} \|_2^2(\b{U}_{M,k}^T\vec{e}_1)(\b{U}_{M,k}^T\vec{e}_1)^T.\]
Then the solution produced by the reference method takes the form \[X_{\mathtt{ref}} = \widecheck{\b{Q}}_M\b{U}_{M,k}Y_{M,k}\b{U}_{M,k}^T\widecheck{\b{Q}}_M^T.\]

\begin{theorem}[Error bound for finite-precision Lanczos with compression] \label{thm:fplancomp}
By the notation and assumptions introduced above, the residual for the approximation $X_{\mathtt{ref}}$ returned by~\Cref{alg: naive}, with the Lanczos process carried out in finite-precision arithmetic, satisfies the following bound:
    \[ \frac{\| AX_{\mathtt{ref}} + X_{\mathtt{ref}}A -\vec{c}\vec{c}^T \|_F}{\|\vec{c} \|_2^2} \leq \frac{\| \vec{\rho}_M \|_F}{\|\vec{c} \|_2^2}
    + C_3\cdot \widetilde{\raterr},
    \]
    with $\vec{\rho}_M$ denoting the residual~\cref{eq:fplanres},
    $C_3 = \frac{2(1+2\epsilon_0)M\lambda_{\max}}{\lambda_{\min} - M^{5/2}\epsilon_2\| A \|_2}$, and 
    \[
    \widetilde{\raterr} = \raterr\big(\vec{\xi}_k, \lambda_{\min} - M^{5/2}\epsilon_2\| A \|_2,\lambda_{\max} + M^{5/2}\epsilon_2\| A \|_2\big);
    \]
    defined according to~\cref{eq:raterr}.
\end{theorem}
\begin{proof}
By the triangle inequality, 
\[
 \| AX_{\mathtt{ref}} + X_{\mathtt{ref}}A -\vec{c}\vec{c}^T \|_F \le 
 \| \vec{\rho}_M \|_F + \|A(X_{\mathtt{ref}}- X_{\mathtt{lan}}) + (X_{\mathtt{ref}}- X_{\mathtt{lan}})A\|_F,
\]
with $X_{\mathtt{lan}} = \widecheck{\b{Q}}_MX_M\widecheck{\b{Q}}_M^T$.
The second term is bounded by
\begin{align*}
 & \|A(X_{\mathtt{ref}}- X_{\mathtt{lan}}) + (X_{\mathtt{ref}}- X_{\mathtt{lan}})A\|_F \\
 =\, &
 \|A\widecheck{\b{Q}}_M (\b{U}_{M,k}Y_{M,k}\b{U}_{M,k}^T - X_M) \widecheck{\b{Q}}_M^T + \widecheck{\b{Q}}_M ( \b{U}_{M,k}Y_{M,k}\b{U}_{M,k}^T - X_M) \widecheck{\b{Q}}_M^T A\|_F \\
 \leq\, & 2\lambda_{\max}\| \widecheck{\b{Q}}_M(\b{U}_{M,k}Y_{M,k}\b{U}_{M,k}^T - X_M)\widecheck{\b{Q}}_M^T\|_F \\
 \leq\, & 2(1+2\epsilon_0)\lambda_{\max}M\| \b{U}_{M,k}Y_{M,k}\b{U}_{M,k}^T - X_M\|_F,
\end{align*}
where the last inequality uses~\cref{eqn: QMnorm}.
The expression
$\b{U}_{M,k}Y_{M,k}\b{U}_{M,k}^T - X_M$ is the approximation error of the rational Krylov method applied to the projected equation
\[ \widecheck{\b{T}}_MX_M + X_M\widecheck{\b{T}}_M = \| \vec{c} \|_2^2\vec{e}_1\vec{e}_1^T.\]
    By~\Cref{lem:dru} and~\cref{eq:spectrumtm},
    \[ \frac{\| \b{U}_{M,k}Y_{M,k}\b{U}_{M,k}^T - X_M\|_F}{\|\vec{c}\|_2^2} \leq \frac{\widetilde{\raterr}}{\lambda_{\min} - M^{5/2}\epsilon_2\| A \|_2},\]
    which completes the proof.
\end{proof}

The result of~\Cref{thm:fplancomp} nearly matches the result of~\Cref{thm:fplanczos}, up to the quantity $\widetilde{\raterr}$, which measures the rational approximation error. When choosing Zolotarev poles, this quantity satisfies the bound~\cref{eqn:zolbound} on slightly enlarged intervals.
This implies that the roundoff error during the Lanczos process has a negligible 
impact on the number of Zolotarev poles needed to attain a certain error, because this number depends logarithmically on the condition number.
\section{Experimental results and comparison with existing algorithms}
\label{sec:numexp}

In this section we present some numerical results to compare our~\Cref{alg: lyapcomp}, which will be named \texttt{compress}, to two existing low-memory variants of the Lanczos method for solving Lyapunov equations: two-pass Lanczos method~\cite{Kre2}, named \texttt{two-pass}, and the compress-and-restart method from~\cite{Kre}, named \texttt{restart}. All algorithms are stopped when the estimated norm of the residual is smaller than $\texttt{tol}\cdot \| \vec{c} \|_2^2$ for some prescribed tolerance \texttt{tol}.

 The \texttt{MATLAB} implementation of \texttt{restart} algorithm we employed is available at \href{https://gitlab.com/katlund/compress-and-restart-KSM}{gitlab.com/katlund/compress-and-restart-KSM}. It uses the Arnoldi method with full reorthogonalization to compute orthonormal bases of Krylov subspaces.
For~\cref{alg: lyapcomp} and two-pass Lanczos method, we employed our own \texttt{MATLAB} implementations available at \href{https://github.com/fhrobat/lyap-compress}{github.com/fhrobat/lyap-compress}. 

To ensure a fair comparison of memory requirements, we store the same number of vectors of length $N$ across all three algorithms. In our practical implementations of \texttt{compress}, the algorithms take as input \texttt{maxmem}, which specifies the maximum number of vectors of size $N$ to be held in memory. Initially, $\texttt{maxmem} - 1$ Arnoldi iterations are performed, after which the residual norm is checked and the poles are computed as described in \Cref{sub: praimp}. Once the required number of poles $k$ is determined, the parameter $m$ is chosen such that $\texttt{maxmem} = m + 2k + 1$. Subsequently, the residual norm is checked and compression in \texttt{compress} is performed every $m$ Lanczos iterations. In all our experiments, \texttt{maxmem} is set to $120$.

The projected Lyapunov equation within \texttt{two-pass} is solved using a rational Krylov subspace method with the same Zolotarev poles used in \texttt{compress}. This results in another, much smaller projected equation, which is solved by diagonalizing the projected matrix.
We emphasize that the extreme eigenvalues of $A$ are needed to determine poles. If the extreme eigenvalues of $A$ are not provided as input, \texttt{two-pass} also performs full orthogonalization during the first $\texttt{maxmem} - 1$ Lanczos iterations and then extracts an approximation of $\lambda_{\min}$ and $\lambda_{\max}$.

All experiments are performed using \texttt{MATLAB} R2021a on a machine \texttt{Intel(R) Core(TM) i5-1035G1 CPU @ 1.00GHz} with $4$ cores and a $8$ GB RAM. 
The Zolotarev poles are computed using \texttt{MATLAB} functions \texttt{ellipke} and \texttt{ellipj} (modified in order to take as input $m$ rather than $1 - m^2$ when to compute elliptic functions of elliptical modulus $\sqrt{1-m^2}$).

All numerical experiments are summarized in tables that include the size of $A$, the prescribed tolerance \texttt{tol}, the number of required poles $k$, the number of matrix-vector products and the computational time for each of the three algorithms compared, and the residual norm of the obtained approximate solutions scaled by $1/\|\vec{c}\|_2^2$ (referred as ``scaled residual'').

\subsection{4D Laplacian}
As a first example, we consider the Lyapunov equation that arises from the centered finite-difference discretization of the 4D Laplace operator on the unit hyper-cube $\Omega = [0,1]^4$ with zero Dirichlet boundary conditions. This results in a matrix $A \in \R^{N \times N}$ where $N$ is a square of a natural number, that corresponds to the discretization of the 2D Laplace operator and takes the form  

\[ A = B \otimes I + I \otimes B, \qquad B = (\sqrt{N}+1)^2\begin{bmatrix}
    2 & -1 & & \\
    -1 & \ddots & \ddots & \\
    & \ddots & \ddots & -1 \\
    & & -1 & 2
\end{bmatrix} \in \R^{\sqrt{N} \times \sqrt{N}}.\]
The vector $\vec{c}$ is chosen as the discretization of the function
\[ f(x,y) = \frac{2}{\pi}\exp\big(-2(x-1/2)^2\big)\exp\big(-2(y-1/2)^2\big)\]
on $[0,1]^2$. 
The matrix $A$ and the vector $\vec{c}$ are then scaled by $1/\|\vec{c}\|_2^2$ and $1/\|\vec{c}\|_2$, respectively. This scaling improves the performance of \texttt{restart}, while the \texttt{compress} and \texttt{two-pass} algorithms behave the same regardless of this transformation.

In this experiment, the extreme eigenvalues of $A$ can be computed analytically and are therefore provided directly as input to the algorithm.

The tolerance for compressing the updated right-hand sides within \texttt{restart} is set to the default tolerance indicated in the \texttt{MATLAB} code, that is, $\texttt{tol} \times10^{-4}$.

In \cref{tab:4Dlap} we compare the three different methods for solving the 4D Laplacian problem. Not surprisingly, the number of matrix-vector products of \texttt{compress} is exactly half the number of matrix-vector products of \texttt{two-pass}. 
For this example, \texttt{restart} struggles to converge, due to the repeated compression of the right-hand side.
The time ratio between \texttt{compress} and \texttt{two-pass} is below $1$, demonstrating the advantage of avoiding a second run of the Lanczos process.
On the other hand, it also stays well above $0.5$ because the compression of the Lanczos basis performed within \texttt{compress} has a non-negligible impact on the execution time. In particular, we observe that as $N$ increases, the time ratio also increases. This is primarily because a larger $N$ results in a greater number of poles $k$, due to the wider spread of the eigenvalues of $A$. Under our fixed maximum memory setting, this leads to a smaller number $m$ of Lanczos iterations between two compression steps. As a result, compressions occur more frequently. Furthermore, in this experiment, the cost of performing a matrix-vector product with $A$ is relatively low, which reduces the advantage of the proposed method over the \texttt{two-pass} method.

\begin{table}[t]
\centering
\begin{filecontents*}{4Dlap_matvec.dat}
N tol k matveccomp matvectwo matvecrest truerescomptwo 

$18\times10^{4}$ $10^{-6}$ 35 658 1316 7031 $5.3\times10^{-7}$
$36\times10^{4}$ $10^{-6}$ 38 936 1872 $>10000$ $5.3\times10^{-7}$
$72\times10^{4}$ $10^{-6}$ 41 1340 2680 $>10000$ $4.1\times10^{-7}$
$144\times10^{4}$ $10^{-6}$ 44 1886 3772 $>10000$ $5.9\times10^{-7}$
\end{filecontents*}
\pgfplotstabletypeset[
  every head row/.style={
    before row=\toprule,after row=\midrule},
  every last row/.style={
    after row=\bottomrule},
display columns/0/.style={column name={$N$}},
display columns/1/.style={column name={\texttt{tol}}},
display columns/2/.style={column name={$k$}},
display columns/3/.style={column name={\makecell[b]{n. matvecs \\ \texttt{compress}
}}},
display columns/4/.style={column name={\makecell[b]{n. matvecs \\ \texttt{two-pass}
}}},
display columns/5/.style={column name={\makecell[b]{n. matvecs \\ \texttt{restart}
}}},
display columns/6/.style={column name={\makecell[b]{scaled residual \\ \texttt{compress} and \\ \texttt{two-pass}
}}},
string type
]
{4Dlap_matvec.dat}

\vspace{1em}

\begin{filecontents*}{4Dlap_time.dat}
N tol k timecomp timetwo timerest timeratio

$18\times10^{4}$ $10^{-6}$ 35 3.7 5.1 $>300$ 0.72
$36\times10^{4}$ $10^{-6}$ 38 10.4 13.9 $>300$ 0.75
$72\times10^{4}$ $10^{-6}$ 41 33.2 40.6 $>300$ 0.82
$144\times10^{4}$ $10^{-6}$ 44 105.5 114.0 $>300$ 0.93
\end{filecontents*}
\pgfplotstabletypeset[
  every head row/.style={
    before row=\toprule,after row=\midrule},
  every last row/.style={
    after row=\bottomrule},
display columns/0/.style={column name={$N$}},
display columns/1/.style={column name={\texttt{tol}}},
display columns/2/.style={column name={$k$}},
display columns/3/.style={column name={\makecell[b]{time \\ \texttt{compress}
}}},
display columns/4/.style={column name={\makecell[b]{time \\ \texttt{two-pass}
}}},
display columns/5/.style={column name={\makecell[b]{time \\ \texttt{restart}
}}},
display columns/6/.style={column name={\makecell[b]{time ratio \\ \texttt{compress}/\texttt{two-pass}
}}},
string type
]
{4Dlap_time.dat}

\caption{Matrix-vector products (top) and execution times (bottom) required to solve the Lyapunov equation arising from the 4D Laplace equation using three different low-memory methods. The scaled residual for \texttt{restart} in the first  row is equal to $9.4 \times 10^{-7}$.}
\label{tab:4Dlap}
\end{table}

\subsection{Model order reduction: Example 1}
This example originates from the \texttt{FEniCS Rail} model\footnote{\url{https://morwiki.mpi-magdeburg.mpg.de/morwiki/index.php/FEniCS_Rail}}:
\[
    \begin{cases}
    E\dot{x}(t) = Mx(t) + Bu(t), \\
    y(t) = Cx(t),
    \end{cases}
\]
where $M, E \in \mathbb{R}^{N \times N}$ are symmetric positive definite matrices and $B \in \mathbb{R}^{N}$. Applying balanced truncation model reduction to this system requires solving a Lyapunov equation of the form
\begin{equation} \label{eq:glyap}
 (-L^{-1}ML^{-T})X + X (-L^{-1}ML^{-T}) = (-L^{-1}B)(-L^{-1}B)^T,
\end{equation}
where $E = LL^T$ is the Cholesky decomposition of $E$.

In practice, the matrix $E$ is first reordered using nested dissection, as implemented in \texttt{MATLAB}, followed by a sparse Cholesky decomposition.

Here, the vector $B$ is chosen as the first column of the input matrix provided by the \texttt{FEniCS Rail} model. Since the norm of $B$ is very small, the tolerance for compression in \texttt{restart} is set to machine precision, denoted by \texttt{eps}. 

In this example, the time ratio between \texttt{compress} and \texttt{two-pass} is close to $0.5$, which is the ratio of matrix-vector products. This is because the matrix-vector product becomes more expensive: applying the matrix $A = -L^{-1}ML^{-T}$ requires a multiplication with $M$ and the solution of two sparse triangular systems, which is computationally intensive. As a result, the execution time of the compression becomes negligible compared to that of the Lanczos process.

Lastly, we note that in the $N = 79{,}841$ case, the scaled residual norm is slightly larger than \texttt{tol}. This is due to a poor estimate of the smallest eigenvalue of $A$ during the first cycle.

\begin{table}[t]
\centering
\begin{filecontents*}{rail_matvec.dat}
N tol k matveccomp matvectwo matvecrest truerescomptwo

5177 $10^{-3}$ 32 669 1338 1428 $5.5\times10^{-4}$
20209 $10^{-3}$ 31 1259 2518 3364 $5.8\times10^{-4}$
79841 $10^{-3}$ 31 2855 5710 $>10000$ $2.4\times10^{-3}$
\end{filecontents*}
\pgfplotstabletypeset[
  every head row/.style={
    before row=\toprule,after row=\midrule},
  every last row/.style={
    after row=\bottomrule},
display columns/0/.style={column name={$N$}},
display columns/1/.style={column name={\texttt{tol}}},
display columns/2/.style={column name={$k$}},
display columns/3/.style={column name={\makecell[b]{n. matvecs \\ \texttt{compress}
}}},
display columns/4/.style={column name={\makecell[b]{n. matvecs \\ \texttt{two-pass}
}}},
display columns/5/.style={column name={\makecell[b]{n. matvecs \\ \texttt{restart}
}}},
display columns/6/.style={column name={\makecell[b]{scaled residual \\ \texttt{compress} and \\ \texttt{two-pass}
}}},
string type
]
{rail_matvec.dat}

\vspace{1em}

\begin{filecontents*}{rail_time.dat}
N tol k timecomp timetwo timerest timeratio

5177 $10^{-3}$ 32 0.6 1 3.8 0.65
20209 $10^{-3}$ 31 6.9 12.9 20.1 0.53
79841 $10^{-3}$ 31 65.9 129.6 $>300$ 0.51
\end{filecontents*}
\pgfplotstabletypeset[
  every head row/.style={
    before row=\toprule,after row=\midrule},
  every last row/.style={
    after row=\bottomrule},
display columns/0/.style={column name={$N$}},
display columns/1/.style={column name={\texttt{tol}}},
display columns/2/.style={column name={$k$}},
display columns/3/.style={column name={\makecell[b]{time \\ \texttt{compress}
}}},
display columns/4/.style={column name={\makecell[b]{time \\ \texttt{two-pass}
}}},
display columns/5/.style={column name={\makecell[b]{time \\ \texttt{restart}
}}},
display columns/6/.style={column name={\makecell[b]{time ratio \\ \texttt{compress}/\texttt{two-pass}
}}},
string type
]
{rail_time.dat}

\caption{Matrix-vector products (top) and execution times (bottom) required to solve the Lyapunov equation arising from the \texttt{FEniCS Rail} model order reduction problem using three different low-memory methods. The scaled residual for \texttt{restart} in the first and second row is $1.0 \times 10^{-3}$ and $7.3 \times 10^{-4}$ respectively.}
\end{table}

\subsection{Model order reduction: Example 2}
This is a variation of the previous example, now using the data from~\cite[Experiment~3]{Ben3}. As before, balanced truncation model reduction is applied to a system of the form
\[
E\dot{T}(t) = \left(M - \sum_{i=1}^t \alpha_i F_i\right)T(t) + Bu(t),
\]
where $E, M, F_i \in \mathbb{R}^{N \times N}$ for each $i$ and $B \in \mathbb{R}^N$,
which leads to a Lyapunov equation of the form~\cref{eq:glyap}. 
The vector $B$ is chosen as the first column of the input matrix, and the matrix $E$ is now diagonal. As in the previous example, the tolerance for compression in \texttt{restart} is set to \texttt{eps}. The coefficients $\alpha_i$ are set to $10$. Note that as $N$ changes, the integer $t$ and the matrices $F_i$ also change, corresponding to different Neumann boundary conditions.

Similarly to the 4D Laplacian, the matrix-vector products with $A$ are computationally efficient due to the diagonal structure of $E$. As a result, \texttt{two-pass} is competitive with \texttt{compress}, since compression steps take a significant amount of time relative to the Lanczos iterations, especially when more poles are required and thus compression occurs more frequently.

\begin{table}[t]
\centering
\begin{filecontents*}{thermal_matvec.dat}
N tol k matveccomp matvectwo matvecrest truerescomptwo 

4813 $10^{-3}$ 45 1337 2674 $>10000$ $4.7\times10^{-4}$
13551 $10^{-3}$ 39 2825 5650 $>10000$ $4.8\times10^{-4}$
25872 $10^{-3}$ 39 4055 8110 $>10000$ $9.7\times10^{-4}$
39527 $10^{-3}$ 37 2189 4378 $>10000$ $4.7\times10^{-4}$
\end{filecontents*}
\pgfplotstabletypeset[
  every head row/.style={
    before row=\toprule,after row=\midrule},
  every last row/.style={
    after row=\bottomrule},
display columns/0/.style={column name={$N$}},
display columns/1/.style={column name={\texttt{tol}}},
display columns/2/.style={column name={$k$}},
display columns/3/.style={column name={\makecell[b]{n. matvecs \\ \texttt{compress}
}}},
display columns/4/.style={column name={\makecell[b]{n. matvecs \\ \texttt{two-pass}
}}},
display columns/5/.style={column name={\makecell[b]{n. matvecs \\ \texttt{restart}
}}},
display columns/6/.style={column name={\makecell[b]{scaled residual \\ \texttt{compress} and \\ \texttt{two-pass}
}}},
string type
]
{thermal_matvec.dat}

\vspace{1em}

\begin{filecontents*}{thermal_time.dat}
N tol k timecomp timetwo timerest timeratio

4813 $10^{-3}$ 45 0.8 0.6 $>8$ 1.3 %1
13551 $10^{-3}$ 39 2.3 3 $>20$ 0.78 %1
25872 $10^{-3}$ 39 5.7 8.1 $>50$ 0.70 %2
39527 $10^{-3}$ 37 4.2 5.8 $>70$ 0.71 %1
\end{filecontents*}
\pgfplotstabletypeset[
  every head row/.style={
    before row=\toprule,after row=\midrule},
  every last row/.style={
    after row=\bottomrule},
display columns/0/.style={column name={$N$}},
display columns/1/.style={column name={\texttt{tol}}},
display columns/2/.style={column name={$k$}},
display columns/3/.style={column name={\makecell[b]{time \\ \texttt{compress}
}}},
display columns/4/.style={column name={\makecell[b]{time \\ \texttt{two-pass}
}}},
display columns/5/.style={column name={\makecell[b]{time \\ \texttt{restart}
}}},
display columns/6/.style={column name={\makecell[b]{time ratio \\ \texttt{compress}/\texttt{two-pass}
}}},
string type
]
{thermal_time.dat}

\caption{Matrix-vector products (top) and execution times (bottom) required to solve the Lyapunov equation arising from the model order reduction problem proposed in~\cite[Experiment $3$]{Ben3} using three different low-memory methods.}
\end{table}

\section{Conclusions}
We have presented a new algorithm for solving large-scale symmetric Lyapunov equations with low-rank right-hand sides. Inspired by previous work~\cite{Cas} on matrix functions, our algorithm performs compression to mitigate the excessive memory required when using a (slowly converging) Lanczos method. Our convergence analysis quantifies the impact of compression on convergence and shows that it remains negligible. Our analysis also quantifies the impact of the loss of orthogonality, due to roundoff error, for both the standard Lanczos method and our new algorithm. Numerical experiments confirm the advantages of compression over existing low-memory Lanczos methods.

\section*{Acknowledgments}
The authors are grateful to Igor Simunec: conversations with him contributed to improving the presentation of the paper.
Part of this work was performed while Francesco Hrobat was staying at EPFL.
Angelo A. Casulli is a member of the INdAM-GNCS research group.  
He has been supported by the National Research Project (PRIN) ``FIN4GEO: Forward and Inverse Numerical Modeling of Hydrothermal Systems in Volcanic Regions with Application to Geothermal Energy Exploitation'' and by the INdAM-GNCS project ``NLA4ML---Numerical Linear Algebra Techniques for Machine Learning.''

\bibliography{references}
\bibliographystyle{siam}

\end{document}